\documentclass{amsart}

\usepackage{amssymb}
\usepackage{amsmath,amsfonts,amsthm,amstext}
\usepackage{amsmath}
\usepackage{mathrsfs}  
\usepackage[all]{xy}
\usepackage{xcolor}

\newtheorem{theorem}{Theorem}[section]
\newtheorem{lemma}[theorem]{Lemma}
\newtheorem{prop}[theorem]{Proposition}
\newtheorem{cor}[theorem]{Corollary}
\newtheorem*{theoremA}{Theorem A}
\newtheorem*{corB1}{Corollary B1}
\newtheorem*{corB2}{Corollary B2}
\newtheorem*{corC}{Corollary C}

\newtheorem*{ConI}{Conjecture I}

\newtheorem*{propD}{Proposition D}
\newtheorem*{propJ}{Proposition J}

\newtheorem*{theoremE1}{Theorem E1}
\newtheorem*{theoremE2}{Theorem E2}
\newtheorem*{theoremF1}{Theorem F1}
\newtheorem*{theoremF2}{Theorem F2}
\newtheorem*{corH}{Corollary H}
\newtheorem*{corG}{Corollary G}

\def\mapnew#1{\smash{\mathop{\longrightarrow}\limits^{#1}}}

\numberwithin{equation}{section}
\def\-{\overline}
\newcommand{\R}{\mathbb{R}}
\newcommand{\Z}{\mathbb{Z}}

\def\X{\mathfrak{X}}

\theoremstyle{definition}

\begin{document}
\title[$\Sigma$-invariants of $\X(G)$]{
On the Bieri-Neumann-Strebel-Renz invariants of the weak commutativity construction $\X(G)$} 
\author{Dessislava H. Kochloukova}
\address{State University of Campinas (UNICAMP), SP, Brazil \\
\newline
email : desi@unicamp.br
\\} 
\email{}
\date{}

\begin{abstract} For a finitely generated group $G$ we calculate the Bieri-Neumann-Strebel-Renz invariant $\Sigma^1(\X(G))$ for the weak commutativity construction $\X(G)$. Identifying $S(\X(G))$ with  $S(\X(G) / W(G))$ we show $\Sigma^2(\X(G),\Z) \subseteq   \Sigma^2(\X(G)/ W(G),\Z)$  and $\Sigma^2(\X(G)) \subseteq $ $  \Sigma^2(\X(G)/ W(G))$ that are equalities when $W(G)$ is finitely generated and we explicitly calculate  $\Sigma^2(\X(G)/ W(G),\Z)$ and $  \Sigma^2(\X(G)/ W(G))$ in terms of the $\Sigma$-invariants of $G$. We calculate completely the $\Sigma$-invariants in dimensions 1 and 2 of the group $\nu(G)$ and show that if $G$ is finitely generated group with finitely presented commutator subgroup then the non-abelian tensor square $G \otimes G$ is finitely presented.
\end{abstract}

\thanks{The author was   partially supported by CNPq grant 301779/2017-1  and by FAPESP grant 2018/23690-6.}

\maketitle

\section{Introduction}

In this paper we consider the Bieri-Neumann-Strebel-Renz $\Sigma$-invariants of the weak commutativity construction $\X(G)$.  By definition $\Sigma^m(G, \mathbb{Z})$ and $\Sigma^m(G)$ are subsets of the character sphere $S(G) = Hom(G, \mathbb{R})/ \sim$ where $\chi_1 \sim \chi_2$ if $\chi_1 \in \mathbb{R}_{>0} \chi_2$. The importance of the $\Sigma$-invariants is that they control which subgroups of $G$ that contain the commutator have homological type $FP_m$ or the homotopical type $F_m$.
  The first $\Sigma$-invariant was defined by Bieri and Strebel in \cite{B-S}, where it was used to classify all finitely presented metabelain groups. In the case of metabelian groups $\Sigma^1(G)$ has strong connection with valuation theory from commutative algebra that was used by Bieri and Groves to prove that the complement of $\Sigma^1(G)$ in the character sphere $S(G)$ is a spherical rational  polyhedron \cite{B-Gr}. 
  In \cite{B-N-S} Bieri, Neumann and Strebel defined the invariant  $\Sigma^1(G)$ for any finitely generated group $G$  and for 3-manifold groups they linked $\Sigma^1(G)$ with the Thurston norm \cite{T}. 
  
  Though in general it is difficult to calculate the $\Sigma$-invariants there are known for several classes of groups though sometime in low dimensions or in specific cases.
  The case of the Thompson group $F$ was considered by Bieri, Geoghegan and Kochloukova \cite{B-G-K} with a geometric proof given in \cite{W-Z} and of generalized Thompson groups $F_{n, \infty}$  by Kochloukova \cite{Ko} and by  Zaremsky  \cite{Z}. The case of free-by-cyclic group was studied by Funke and Kielak  \cite{B-F}, \cite{DK}, Cashen and Levitt \cite{C-L} and of Poincare duality group of dimension 3 by Kielak \cite{DK}. In \cite{D-K-L} Dowdall, Kapovich and Leininger consider links between dynamical properties of the expanding action of $\mathbb{Z}$ on a free finite rank free group and $\Sigma^1$. Though the case of right angled group was completely resolved by Meier, Meinert and van Wyk  \cite{Meinert-VanWyk} the case of general Artin groups is widely open. Still there are some results on $\Sigma^1(G)$ for specific Artin groups
  by Almeida \cite{Kisnney}, Almeida and Kochloukova \cite{Kisnney-K}. In  \cite{A-L} Almeida and Lima calculated $\Sigma^1(G)$ for Artin group of finite type ( i.e spherical type). The case of combinatorial wreath product was considered by Mendon\c{c}a \cite{Luis}.
  
  Let $G$ be a group
and $\-G$ an isomorphic copy of the group $G$.
 The group $\X(G)$ was defined by Sidki in \cite{Said} by the presentation 
  $${\X}(G) = \langle G,  \overline{G} \mid [g, \-g] = 1  \hbox{ for } g \in G \rangle,$$
  where $\-g$ is the image of $g$ in $\-G$.  In \cite[Thm.~C]{Said} Sidki proved  that if $G \in {\mathcal P}$  then ${\X}(G) \in {\mathcal P}$ when ${\mathcal P}$  is one of the following classes of groups : finite $\pi$-groups, where $\pi$ is a set of primes; finite nilpotent groups; solvable groups and perfect groups.
  Later 
  the following classes of groups ${\mathcal P}$ were added to the above list :
  finitely generated nilpotent groups by 
  Gupta, Rocco and Sidki  \cite{GRS}, virtually polycyclic groups by Lima and Oliveira \cite{LimaOliveira}, soluble groups of homological type $FP_{\infty}$ by Kochloukova and Sidki \cite{KochSidki}, finitely resented groups by Kochloukova and Bridson \cite{BK},  finitely generated  virtually nilpotent groups \cite{BK2}, finitely generated Engel groups \cite{BK2}. In this paper we add to the above list the class of finitely generated groups for which $\Sigma^1(G) ^c = S(G) \setminus \Sigma^1(G)$ is a  rationally defined spherical polyhedron. By definition a subset of $S(G)$ is a  rationally defined spherical  polyhedron if it is a finite union of finite intersection of closed rationally defined semispheres in $S(G)$, where rationality means that the semishere is defined by a rational vector. 
  
  The group ${\X}(G)$ has a special normal abelian subgroup $W(G)$ such that $\X(G) / W(G)$ is a subdirect product of $G \times G \times G$ that  maps surjectively on pairs. 
  Our first result calculates $\Sigma^1(\X(G))$ for any finitely generated group $G$.
 \begin{theoremA}
 	Let $G$ be a finitely generated group and $\chi : \X(G) \to \mathbb{R}$ be a character. Consider the homomorphisms $\chi_1, \chi_2 : G \to \R$, where $\chi_1(g) = \chi(g)$ and $\chi_2(g) = \chi(\overline{g})$.
 	 Then
 	$[\chi] \in \Sigma^1(\X(G))$ if and only if one of the following holds :
 	
 	1. $\chi_1 \not= 0$ , $ \chi_2 \not= 0$ and $ \chi_1 \not= \chi_2$;
 	
 	2. $\chi_1 = 0 $, $ [\chi_2] \in \Sigma^1(G)$;
 	
 	3. $\chi_2 = 0 $, $ [\chi_1] \in \Sigma^1(G)$;
 	
 	4. $\chi_1 = \chi_2 \not= 0$ and $[\chi_1] \in \Sigma^1(G)$.
 	
 	In particular identifying $S(\X(G))$ with $S(\X(G)/ W(G))$ via the epimorphism $\X(G) \to \X(G)/ W(G)$ we have the equality  $\Sigma^1(\X(G)) = \Sigma^1(\X(G)/ W(G))$.
 \end{theoremA}
Theorem A implies immediately the following corollary.
\begin{corB1} Let $G$ be a finitely generated group.
Then $\Sigma^1(G) ^c$ is a  rationally defined spherical  polyhedron if and only if  $\Sigma^1(\X(G))^c$ is a  rationally defined spherical  polyhedron.
\end{corB1} 

Let $\pi_1 : \X(G) \to G$ be the epimorphism that sends $\overline{G} $ to 1 and is identity on $G$. Let $\pi_2 : \X(G) \to {G}$ be the epimorphism that sends ${G} $ to 1 and sends $\overline{g}$ to $g$ for every $g \in G$. Finally set the  epimorhism
$
\pi_3 : \X(G) \to  G
$
 that sends both $g$ and $\overline{g}$ to $g$ for every $g \in G$.

\begin{corB2}   Let $G$ be a finitely generated group and
$N$ be a subgroup of $\X(G)$ that contains the commutator subgroup $\X(G)'$. Then $N$ is finitely generated if and only if $\pi_1(N)$, $\pi_2(N)$ and $\pi_3(N)$ are all finitely generated. 
In particular $\X(G)'$ is finitely generated if and only if $G'$ is finitely generated. 
\end{corB2}

Limit groups were defined by Sela and independently were studied by Kharlampovich and Myasnikov that refered to them as fully residually free groups.  Limit groups were discovered  during the development of the theory that lead to the solution of the Tarski problem on the elementary theory of non-abelian free groups of finite rank in \cite{K-M}  and \cite{Sela}. In \cite{Desi1} Kochloukova showed that for a non-abelian limit group $G$ we have $\Sigma^1(G) = \emptyset$.

\begin{corC} Let $G$ be a non-abelian limit group. Then
	$$\Sigma^1(\X(G)) = \{ [\chi] \in S(\X(G)) \mid \chi_1 \not= 0, \chi_2 \not= 0, \chi_1 \not= \chi_2 \}.$$
	\end{corC}

We continue with the study of the $\Sigma$-invariants focusing on $\Sigma^2(\X(G), \Z)$ and $\Sigma^2(\X(G))$. Our first result in this direction is for non-abelian limit groups $G$.

\begin{propD}
		Let $G$ be a finitely generated group.
		Suppose that $\Sigma^1(G) = \emptyset$. Then $\Sigma^2(\X(G), \mathbb{Z}) = \emptyset$. In particular this holds for non-abelian limit groups $G$.
	\end{propD}
Recall that the group ${\X}(G)$ has a special normal abelian subgroup $W(G)$ such that $\X(G) / W(G)$ is a subdirect product of $G \times G \times G$ that  maps surjectively on pairs. Thus together with the result of Bridson, Howie, Miller, Short \cite[Thm.~A]{BHMS} implies that  
$\X(G) / W(G)$ is finitely presented whenever $G$ is finitely presented.  A homological version of this result was proved in \cite[Thm.~D]{KochSidki} using $\Sigma$-theory  i.e.  $\X(G) / W(G)$ is $FP_2$ whenever $G$ is $FP_2$. 
Recently Bridson and Kochloukova generalised this by showing in \cite{BK} that $\X(G)$ is finitely presented (resp. $FP_2$) if and only if $G$ is finitely presented (resp. $FP_2$).  This does not generalise to homological type $FP_3$ since for finitely generated free non-cyclic group $G$ the group $\X(G)$ is not of type $FP_3$ \cite{BK}.

In a recent work Kochloukova and Lima \cite{KL2} studied the $\Sigma$-invarants of subdirect products of non-abelian limit groups, in particular this applies for $\X(G) / W(G)$ when $G$ is a non-abelian limit group. In the following theorem the groups are not presumed limit groups and different tecniques from \cite{KL2}  are applied.   
\begin{theoremE1}
Let $G$ be of type $FP_2$ and $\chi : K = \X(G) / W(G) \to \R$ be a character. Let $\chi_1, \chi_2 : G \to \R$  be characters defined by $\chi_1(g) =  \chi(g)$ and $\chi_2(g) = \chi(\overline{g})$. 
Then $[\chi] \in \Sigma^2(K, \Z)$  if and only if one of the following cases holds :

1. $\chi_1  = 0$, $[\chi_2] \in \Sigma^2(G, \Z)$;

2. $\chi_2 = 0$, $[\chi_1] \in \Sigma^2(G, \mathbb{Z})$;

3. $\chi_1 = \chi_2 \not= 0$ and $[\chi_1] \in \Sigma^2(G, \Z)$;

4. $\chi_1 \not= 0, \chi_2 \not= 0$, $\chi_1 \not= \chi_2$ and one of the following holds :

4.1. $\{ [\chi_1], [\chi_2] \} \subseteq \Sigma^1(G)$;

4.2. $\{ [\chi_1], [\chi_1 - \chi_2] \} \subseteq \Sigma^1(G)$;

4.3. $\{ [\chi_2], [\chi_2 - \chi_1] \} \subseteq \Sigma^1(G)$.

\end{theoremE1}
The proof of Theorem E1 uses substantially  Theorem \ref{thmI}.
Since Theorem \ref{thmI} has  homotopical version we have a homotopical version of Theorem E1.

\begin{theoremE2}
Let $G$ be finitey presented group and $\chi : K = \X(G) / W(G) \to \R$ be a character. Let $\chi_1, \chi_2 : G \to \R$  be characters defined by $\chi_1(g) =  \chi(g)$ and $\chi_2(g) = \chi(\overline{g})$. 
Then $[\chi] \in \Sigma^2(K)$  if and only if one of the following cases holds :

1. $\chi_1  = 0$, $[\chi_2] \in \Sigma^2(G)$;

2. $\chi_2 = 0$, $[\chi_1] \in \Sigma^2(G)$;

3. $\chi_1 = \chi_2 \not= 0$ and $[\chi_1] \in \Sigma^2(G)$;

4. $\chi_1 \not= 0, \chi_2 \not= 0$, $\chi_1 \not= \chi_2$ and one of the following holds :

4.1. $\{ [\chi_1], [\chi_2] \} \subseteq \Sigma^1(G)$;

4.2. $\{ [\chi_1], [\chi_1 - \chi_2] \} \subseteq \Sigma^1(G)$;

4.3. $\{ [\chi_2], [\chi_2 - \chi_1] \} \subseteq \Sigma^1(G)$.

\end{theoremE2}

In general little is known for $W(G)$. In \cite{KochSidki} Kochloukova and Sidki proved using homological methods that if $G$ is $FP_2$ and  $G'/ G''$ is finitely generated then $W(G)$ is finitely generated. In Theorem F1 we get partial information on $\Sigma^2(\X(G), \mathbb{Z})$.

\begin{theoremF1} Suppose  $G$ is a group of type $FP_2$. 	For a character $\chi : \X(G) \to \R$ define $\widehat{\chi}: \X(G) / W \to \R$ to be the  character induced by $\chi$, where   $W = W(G)$.
	
a) If 	
$ [\chi] \in \Sigma^2(\X(G), \Z)$ then  $[\widehat{\chi}] \in \Sigma^2(\X(G)/W, \Z)$.

b) Suppose $W$ is finitely generated and $[\widehat{\chi}] \in \Sigma^2(\X(G)/W, \Z)$. Then $ [\chi] \in \Sigma^2(\X(G), \Z)$.
In particular this holds when the abelianization of the commutator group $G'= [G,G]$ is finitely generated.
\end{theoremF1}
Theorem F1 has the following homotopical version.

\begin{theoremF2} Suppose  $G$ is a finitely presented group. 	For a character $\chi : \X(G) \to \R$ define $\widehat{\chi}: \X(G) / W \to \R$ to be the  character induced by $\chi$, where   $W = W(G)$.
	
a) If 	
$ [\chi] \in \Sigma^2(\X(G))$ then  $[\widehat{\chi}] \in \Sigma^2(\X(G)/W)$.

b) Suppose $W$ is finitely generated and $[\widehat{\chi}] \in \Sigma^2(\X(G)/W)$. Then $ [\chi] \in \Sigma^2(\X(G))$.
In particular this holds when the abelianization of the commutator group $G'= [G,G]$ is finitely generated.
\end{theoremF2}

The following result follows from Theorem A, Theorem E1, Theorem E2, Theorem F1 and Theorem F2. Though it is an easy corollary of those theorems  it looks more symmetric than the theorems it derives from. For a group $H$ and a subgroup $N$ of $H$ by definition $S(H,N) = \{ [\chi] \in S(H) ~| ~\chi(N) = 0 \}$.

\begin{corG}
a) Suppose that $G$ is a finitely generated group. Then we have the disjoint union $$\Sigma^1(\X(G))^c = V_1 \cup V_2 \cup V_3$$
where
$$
V_i = \Sigma^1(\X(G))^c \cap S(\X(G), Ker (\pi_i))
$$
and the epimorphisms $\pi_i : \X(G) \to G$ induces a bijection
$\pi_i^* :  \Sigma^1(G) ^c \to V_i$;

b) Suppose that $G$ is a finitely presented group. Then
$$
W_1 \cup W_2 \cup W_3 \cup ( V_1 + V_2) \cup (V_2 + V_3) \cup (V_1 + V_3) \subseteq \Sigma^2(\X(G))^c$$
where $$
W_i = \Sigma^2(\X(G))^c \cap S(\X(G), Ker (\pi_i))
$$
and the epimorphisms $\pi_i : \X(G) \to G$ induces a bijection
$\pi_i^* : \Sigma^2(G)^c \to W_i$;

c) Suppose that $G$ is of homological type $FP_2$. Then
$$
M_1 \cup M_2 \cup M_3 \cup ( V_1 + V_2) \cup (V_2 + V_3) \cup (V_1 + V_3) \subseteq \Sigma^2(\X(G), \mathbb{Z})^c$$
where $$
M_i = \Sigma^2(\X(G), \mathbb{Z})^c \cap S(\X(G), Ker (\pi_i))
$$
and the epimorphisms $\pi_i : \X(G) \to G$ induces a bijection
$\pi_i^* : \Sigma^2(G, \mathbb{Z})^c \to M_i$;
\end{corG}

As a corollary of Theorem F1 and Theorem F2 we obtain the following result.

\begin{corH} Let $G$ be a group of type $FP_2$. (resp finitely presented).	For a subgroup $N$ of $\X(G)$ that contains the commutator $\X(G)'$ and such that $N$ is of type $FP_2$ (resp. finitely presented) we have that $N/ W(G)$ is of type $FP_2$ too (resp. finitely presented).
	Furthermore for $m \geq 1$ the commutator $\X(G)'$ is $FP_m$ (resp. finitely presented) if and only if the commutator $G'$ is $FP_m$ (resp. finitely presented). When this happens $W(G)$ is finitely generated.
\end{corH}

Inspired by Theorem F1 and Theorem F2 we suggest the following conjecture.

\begin{ConI}  Let $G$ be a group of type $FP_2$ ( resp. finitely presented). Then identifying $S(\X(G))$ with $S(\X(G)/ W(G))$ via the epimorphism $\X(G) \to \X(G)/ W(G)$ we have the equality $$\Sigma^2(\X(G), \Z) = \Sigma^2(\X(G)/ W(G),\Z) \hbox{ (resp. } \Sigma^2(\X(G)) = \Sigma^2(\X(G)/ W(G)) ).$$
\end{ConI}

Note that by Theorem F1 and Theorem F2 we have that Conjecture I holds when $G'/ G''$ is finitely generated. 
Note that Theorem E1 implies that
for a non-abelian limit group $G$ we have $\Sigma^2(\X(G) / W(G), \Z) = \emptyset$, so the Conjecture I holds for non-abelian limit groups. Another way to state the Conjecture I is that in Corollary G in parts b) and c) the inclusions are equalities.

In Section \ref{section-tensor} we consider the non-abelian tensor square $G \otimes G$ of a group $G$ and the contruction $\nu(G)$. In  \cite{Rocco} Rocco defined for an arbitrary group $G$ the  group $\nu(G)$. In \cite{Ellis} Ellis and  Leonard studied a similar construction. The construction $\nu(G)$ is strongly related to the construction $\X(G)$ in the following way : there is a central subgroup $\Delta$ of $\nu(G)$ such that
$$\nu(G) / \Delta \simeq \X(G)/ R(G)$$
where $R(G)$ is a normal subgroup of $\X(G)$ such that $W(G)/ R(G) \simeq H_2(G, \mathbb{Z})$.
By \cite{Rocco}  the non-abelian tensor square $G \otimes G$ is isomorphic to the subgroup $[G, \overline{G}]$ of $\nu(G)$. 
Using properties of $\X(G)/ W$ we show the following result.

\begin{propJ}
Let $G$ be a finitely generated group such that the commutator subgroup $G'$ is finitely presented (resp. is $FP_2$). Then the non-abelian tensor square $G \otimes G$ is finitely presented (resp. is $FP_2$).
\end{propJ}

In section \ref{section-tensor} we determine the invariants $\Sigma^1(\nu(G))$, $\Sigma^2(\nu(G), \mathbb{Z})$ and $\Sigma^2(\nu(G))$ and identify them with the corresponding invariants of $\X(G)/ W$.

\section{Preliminaries on the $\Sigma$-invariants}

In \cite{Wall} Wall defined a group $G$ to be of  homotopical type $F_n$  if there is a classifying space $K(G,1)$ with finite $n$-skeleton. The homotopical type $F_2$ coincides with  finite presentability (in terms of generators and relations). A homological version of this property, called $FP_n$, was defined by Bieri in \cite{Bieribook}.
A group $G$ is of homological type $FP_n$ if the trivial $\mathbb{Z} G$-module $\mathbb{Z}$ has a projective resolution with all modules finitely generated in dimension $\leq n$.

  Higher dimensional homological invariants $\Sigma^n(G, A)$ for a $\Z G$-module $A$ were defined by Bieri and Renz in \cite{B-Renz}, where they showed  that   $\Sigma^n(G, \Z)$ controls which subgroups of $G$ that contain the commutator are of homological type $FP_n$. In \cite{Renzthesis} Renz defined  the higher dimensional homotopical invariant $\Sigma^n(G)$ for groups $G$ of homotopical type $F_n$ and similar to the homological case $\Sigma^n(G)$ controls the homotopical finiteness properties of the subgroups of $G$  above the commutator. In all cases the $\Sigma$-invariants are open subsets of the character sphere $S(G)$.
For a group $G$ of type $F_n$ we have $\Sigma^n(G) = \Sigma^n (G, \Z) \cap \Sigma^2(G)$.
The description of the $\Sigma$-invariants of right-angled Artin groups by Meier, Meinert and  Van Wyk shows  that   the inclusion $\Sigma^n(G) \subseteq \Sigma^n(G, \mathbb{Z})$ is not necessary an equality for $n \geq 2$ \cite{Meinert-VanWyk}.
For the homotopical invariants we note that $\Sigma^1(G) = \Sigma^1(G, \Z)$ and in general $\Sigma^n(G) \subseteq \Sigma^n(G, \Z)$.

By definition a character $\chi : G \to \R$ is a non-zero homomorphism and $\Sigma^n(G, \Z)$ is a subset of the character sphere  $S(G)$.  The character sphere $S(G)$ is the set of equivalence classes $[\chi]$ of characters $\chi : G \to \R$, where 
two characters $\chi_1$ and $\chi_2$ are equivalent if one is obtained from the other by multiplication with any positive real number. For a fixed character $\chi : G \to \R$ define
$$
G_{\chi} = \{ g \in G \mid \chi(g) \geq 0 \}.
$$
Recall that for an associative ring $R$ and $R$-module $A$ we say that $A$ is of type $FP_n$ over $R$ if $A$ has a projective resolution over $R$ where all projectives in dimension up to $n$ are finitely generated i.e. there is an exact complex
$$
{\mathcal P} : \ldots \to P_i \to P_{i-1} \to \ldots \to P_0 \to A \to 0,
$$
where each $P_j$ is a projective $R$-module and for $i \leq n$ we have that $P_i$ is finitely generated.

 Let  $D$ be an integral domain.
By definition for a (left)  $DG$-module $A$ $$\Sigma^n_D(G, A) = \{[\chi] \in S(G) \mid A \textrm{ is of type $FP_n$ as $DG_{\chi}$-module}\}.$$ When $A$ is the trivial (left) $DG$-module $D$, we denote by $ \Sigma^n(G, D)$ the invariant $ \Sigma^n_D(G, D)$.

  Note that if $\Sigma^n(G, \Z)$ is non-empty set then $G$ is $FP_n$ i.e. is finitely generated.  Later we will need the description of $\Sigma^1(G)$ given by the Cayley graph of  a finitely generated group $G$. Let $X$ be a finite generating set of $G$. Consider the Cayley graph $\Gamma$ of $G$ associated with the generating set $X$ i.e. the set of vertices is $V(\Gamma) = G$ and the set of edges is  $E(\Gamma) = X \times G$ with the edge $e = (x, g)$ having beginning $g$ and end $g x$. The group $G$ acts on $\Gamma$ via left multiplication on $V(\Gamma)$ and $h. e = (x, hg) $ for any $h \in G$. The letter $x$ is called the label of the edge $e$ and we write $(x^{-1} , gx)$ for the inverse of $e$ and call $x^{-1}$ the label of $e^{-1}$. For a fixed character $\chi : G \to \R$ we write $\Gamma_{\chi}$ for the subgraph of $\Gamma$ spanned by the vertices in $G_{\chi}$. By definition
$$
\Sigma^1(G) = \{ [\chi] \in S(G) \mid \Gamma_{\chi} \hbox{ is  a connected graph}\}.
$$ 
Suppose now that $G$ is finitely presented with a finite presentation $\langle X | R \rangle$. Gluing to the Cayley graph $\Gamma$ 2-cells at every vertex that spell out the relations of $R$ we get that Cayley complex $\mathcal{C}$  associated to the above finite presentation. For a fixed character $\chi : G \to \R$ we write $\mathcal{C}_{\chi}$ for the subcomplex of $\mathcal{C}$ spanned by the vertices in $G_{\chi}$. By definition
$$
\Sigma^2(G) = \{ [\chi] \in S(G) \mid \hbox{ there is a finite presentation for which } \mathcal{C}_{\chi} \hbox{ is  1-connected }\}.
$$ 

The first result is folclore, it is an obvious corollary of the fact that $\Sigma^1(G, \Z) = \Sigma^1(G)$ and tensoring is right exact functor.

\begin{lemma} \label{quotient-sigma1} Let $\pi : G_1 \to G_2$ be an epimorphism of finitely generated groups, $\mu_2 : G_2 \to \mathbb{R}$ be a character (i.e. non-zero homomorphism) and $\mu_1 = \mu_2 \circ \pi$. Suppose that $[\mu_1] \in \Sigma^1(G_1)$. Then $[\mu_2] \in \Sigma^1(G_2)$.
\end{lemma} 

We warn the reader that the previous lemma does not hold for $\Sigma^2$.

\begin{theorem} \cite[Thm.~ 9.3]{Meinert-VanWyk} \label{teo-Sigma.de.gr.=.Sigma.de.subgr.}
	Let $H$ be a subgroup of $G$, $M$ be a $DG$-module and $\xi: G \rightarrow \R$ be a character. If $[G:H] < \infty$ then $$[\xi|_H] \in \Sigma^n_D(H,A) \Leftrightarrow [\xi] \in \Sigma^n_D(G,A).$$
	In particular, if $n=0$, then $$A \textrm{ is a finitely generated } DG_{\xi} \textrm{-module } \Leftrightarrow A \textrm{ is a finitely generated } DH_{\xi|_H} \textrm{-module}.$$
\end{theorem}

In \cite{B-G} Bieri and Geoghegan proved  a formula for the homological  invariants $\Sigma^n( - , F)$ for a direct product of groups, where $F$ is the trivial module and is a field.  If $F$ is substituted with the trivial module $\Z$ the result is wrong in both homological and homotopical settings provided the dimension is sufficiently high, see \cite{Meinert-VanWyk} and \cite{Schutz}. 

\begin{theorem}  {\bf Direct product formula} \cite[Thm.~ 1.3, and Prop.~ 5.2]{B-G} \label{teo-conjec.prod.dir.corpo}
	Let $n \geq 0$ be an integer, $G_1, G_2$ be finitely generated groups and $F$ be a field. Then, $$\Sigma^n(G_1 \times G_2, F)^c = \displaystyle\bigcup_{p=0}^n \Sigma^p(G_1, F)^c * \Sigma^{n-p}(G_2, F)^c,$$ where $*$ denotes the join of subsets  of the character sphere $S(G_1 \times G_2)$ and $^c$ denotes the set-theoretic complement of subsets of a suitable character sphere.
\end{theorem}

The above theorem means that if $\mu : G_1 \times G_2 \to \R$ is a character with $\mu_1 = \mu \mid_{G_1}$ and  $\mu_2 = \mu \mid_{G_2}$ then $[\mu] \in \Sigma^n(G_1 \times G_2, F)^c = S(G_1 \times G_2, F) \setminus \Sigma^n(G_1 \times G_2, F)$ precisely when one of the following conditions hold :  

1. $\mu_1 \not=0$, $\mu_2 \not= 0$ and $[\mu_1] \in \Sigma^p(G_1, F)^c = S(G_1) \setminus \Sigma^p(G_1, F), [\mu_2] \in \Sigma^{n-p} (G_2, F)^c = S(G_2) \setminus \Sigma^{n-p}(G_2, F)$ for some $0 \leq p \leq n$;

or 

2. one of the characters $\mu_1$, $\mu_2$ is trivial and for the non-trivial one, say $\mu_i$, we have $[\mu_i] \in \Sigma^n(G_i, F)^c = S(G_i) \setminus \Sigma^n(G_i, F)$.

Though Theorem \ref{teo-conjec.prod.dir.corpo} does not hold in general when $F$ is not a field, it holds in small dimensions $1 \leq n \leq 2$.

\begin{theorem} \cite{G}  \label{teo-conjec.prod.dir.corpo2}
	Let $1 \leq n \leq 2$ be an integer, $G_1, G_2$ be finitely generated groups. Then, $$\Sigma^n(G_1 \times G_2, \Z)^c = \displaystyle\bigcup_{p=0}^n \Sigma^p(G_1, \Z)^c * \Sigma^{n-p}(G_2, \Z)^c,$$ where $*$ denotes the join of subsets of the character sphere  $S(G_1 \times G_2)$ and $^c$ denotes the set-theoretic complement of subsets of a suitable character sphere.
\end{theorem}

\begin{theorem} \label{BRhomotopic} \cite{B-Renz}, \cite{Renzthesis} Let $G$ be a group of type $F_n$ (resp. $FP_n$) and $N$ be a subgroup of $G$ that contains the commutator subgroup $G'$. Then $N$ is of type $F_n$ (resp. $FP_n$) if and only if
	$$S(G,N) = \{ [\chi] \in S(G) \mid \chi(N) = 0 \} \subseteq \Sigma^n(G) \ \ (\hbox{resp. } \Sigma^n(G,\Z)) .$$
\end{theorem}
Bieri and Renz proved the
 homological version of Theorem \ref{BRhomotopic} in \cite{B-Renz}. The homotopical version for $m = 2$ was proved by Renz in \cite{Renzthesis} and the general homotopical case $m \geq 3$ follows from the formula $\Sigma^m(G) = \Sigma^m(G, \mathbb{Z}) \cap \Sigma^2(G)$.

The following theorem can be traced back to several papers : Gehrke results in \cite{G}; the Meier, Meinert and wanWyk description of the $\Sigma$-invariants for right angled Artin groups \cite{Meinert-VanWyk} or the Meinert result on the $\Sigma$-invariants for direct products of virtually free groups \cite{Me2}.

\begin{theorem} \label{new111} \cite{G}, \cite{Meinert-VanWyk}, \cite{Me2} Let $F_2$ be the free group on 2 generators. If $\chi : F_2^s = F_2 \times \ldots \times F_2 \to \mathbb{R}$ is a character whose restriction on precisely $n$ copies of $F_2$ is non-zero, then 	$[\chi] \in \Sigma^{n-1}(F_2^s) \setminus \Sigma^{n}(F_2^s)  $. 
\end{theorem}

Recently the following result was obtained by Kochloukova and Mendon\c{c}a. It should be viewed as a monoidal version of Theorem \ref{BRhomotopic}. It is surprising it has not be discovered earlier as Theorem \ref{BRhomotopic} is quite well-known and the proof of  Theorem \ref{thmI} in \cite{K-Me} is based on ideas from the proof Theorem \ref{BRhomotopic} but is slightly more tecnical.

\begin{theorem} \label{thmI}  \cite{K-Me} a) Let $[H,H] \subseteq K \subseteq H$ be groups such that $H$ and $K$ are of type $FP_n$. Let $\chi : K \to \R$ be a character such that $\chi([H,H]) = 0$. Then $[\chi] \in \Sigma^n(K, \Z)$ if and only if $[\mu] \in \Sigma^n(H, \Z)$ for every character $\mu : H \to \R$ that extends $\chi$.

b) Let $[H,H] \subseteq K \subseteq H$ be groups such that $H$ and $K$ are finitely presented. Let $\chi : K \to \R$ be a character such that $\chi([H,H]) = 0$. Then $[\chi] \in \Sigma^2(K)$ if and only if $[\mu] \in \Sigma^2(H)$ for every character $\mu : H \to \R$ that extends $\chi$.
	\end{theorem}

\section{Preliminaries on subdirect products, limit groups, \iffalse{the Virtual Surjection Conjecture and the Monoidal Virtual Surjection  Conjecture}\fi} \label{prel}

The class of limit groups contains all finite rank free groups and the orientable surface groups. It coincides with the class of the fully residually free groups $G$ i.e. for every finite subset $X$ of $G$ there is free group $F$ and a homomorphism $\varphi : G \to F$ whose restriction on $X$ is injective. Limit groups are of homotopical type $F$ i.e. $FP_{\infty}$, finite presentability and finite cohomological dimension.

A subgroup $G \subseteq G_1 \times \ldots \times G_m$ is a subdirect product if the projection map $p_i : G \to G_i$ is surjective for all $ 1 \leq i \leq m$. Denote by $p_{i_1, \ldots, i_n} : G \to G_{i_1} \times \ldots \times G_{i_n}$ the projection map that sends $(g_1, \ldots, g_m)$ to $(g_{i_1}, \ldots, g_{i_n})$.

\begin{theorem} \cite{Desi1} Let $G \subseteq G_1 \times \ldots \times G_m$  be a subdirect product of non-abelian limit groups $G_1, \ldots, G_m$ such that $G \cap G_i \not= 1$ for every $1 \leq i \leq m$. Then if $G$ is of type $FP_n$ for  some $n \leq m$ then $p_{i_1, \ldots, i_n}(G)$ has finite index in $G_{i_1} \times \ldots \times G_{i_n}$ for every $ 1 \leq i_1 < \ldots < i_n \leq m$.
\end{theorem}

\begin{theorem} \cite{Desi1} \label{limit0} Let $G$ be a non-abelian limit group. Then $\Sigma^1(G) = \emptyset$.
	\end{theorem}

 The  following conjecture  was defined by Kuckuck in \cite{Benno}.

\medskip
\noindent
{\bf The Virtual Surjection Conjecture} \cite{Benno} {\it Let $G \subseteq G_1 \times \ldots \times G_m$  be a subdirect product of  groups $G_1, \ldots, G_m$ such that $G \cap G_i \not= 1$ for every $1 \leq i \leq m$ and each $G_i$ is of homotopical type $F_n$ for some $n \leq m$. Suppose that $p_{i_1, \ldots, i_n}(G)$ has finite index in $G_{i_1} \times \ldots \times G_{i_n}$ for every $ 1 \leq i_1 < \ldots < i_n \leq m$. Then $G$ is of type $F_n$. }

\medskip The motivation behind the 
the Virtual Surjection Conjecture is that it holds for $n = 2$ \cite{BHMS} and was established by Bridson, Howie, Miller and Short as a corollary of the 1-2-3 Theorem. Furthermore the Virtual Surjection Conjecture holds for any $n$ when $G$ contains $G_1' \times \ldots \times G_m'$ \cite{Benno}. A homological version of the Virtual Surjection Conjecture was sugested in \cite{KL1} and proved for $n = 2$.

\begin{theorem} \cite{KL1}  \label{homological1-2-3} Let  $G \subseteq G_1 \times \ldots \times G_m$  be a subdirect product of  groups $G_1, \ldots, G_m$ such that $G \cap G_i \not= 1$ for every $1 \leq i \leq m$ and each $G_i$ is of homological type $FP_2$. Suppose that $p_{i_1, i_2}(G)$ has finite index in $G_{i_1} \times G_{i_2}$ for every $ 1 \leq i_1 < i_2 \leq m$. Then $G$ is of type $FP_2$.
	\end{theorem}

\section{Preliminaries on $\X(G)$} \label{prel-X(G)}

Recall that $${\X}(G) = \langle G,\- G \mid [g, \-g] = 1  \hbox{ for } g \in G \rangle,$$
where $\-G$ is an isomorphic copy of the group $G$ and  $\-g$ is the image of $g \in G$ in $\-G$.
In \cite{Said} was defined the normal subgroup $$L = L(G) = \langle \{ \-g^{-1}g  \mid g \in G \}  \rangle$$  of $\X(G)$. Note that  $$\X(G) = L \rtimes G.$$ In \cite{LimaOliveira} Lima and Oliveira showed that  the abelianization $L/ L'$ is finitely generated  whenever $G$ is finitely generated. In \cite{BK} Bridson and Kochloukova generalized this by showing that when $G$ is finitely generated, $L$ is finitely generated.

Another important normal subgroup of $\X(G)$ is $$D= D(G) = [G, \-G].$$ There are canonical epimorphisms of groups
$$
\X(G) \to \X(G)/ D \simeq G \times \-G \simeq G \times G
$$
and
$$\X(G) \to \X(G) / L \simeq G.
$$
The diagonal map of this two epimorphisms induces a map
$$
\rho : \X(G) \to G \times G \times G
$$
with kernel $$W = W(G) = L(G) \cap D(G).$$ This map after some permutation of the factors $G$ in $G \times G \times G$ can be explicitely given by
$$ \rho(g) = (g,g,1) \hbox{ and } \rho(\-g) = (1,g,g).$$ 
Note that
$$
Im (\rho) = \{ (g_1, g_2, g_3) \mid g_1 g_2^{-1} g_3 \in G' \}
$$
is a subdirect product of $ G \times G \times G$ that maps surjectively on  pairs and contains the commutator subgroup $G' \times G' \times G'$.
The defining relations of $\X(G)$ easily imply that 
$$
[L, D] = 1$$
and this property is crucial to develop structure theory for $\X(G)$. For example it implies that $W(G)$ is an abelian group. Furthermore $W(G)$ can be viewed as $\X(G)$-module via conjugation with $DL$ acting trivially. Thus $W(G)$ is a $\X(G)/ DL$-module and $\X(G)/ DL \simeq G/ G'$ is abelian.

The following result was proved by Kochloukova and Sidki in \cite{KochSidki} using homological techniques. Note that every finitely presented group  is $FP_2$.
 
 \begin{theorem} \cite{KochSidki} \label{commutator}  If $G$ is of homological type $FP_2$ and $G'/ G''$ is finitely generated then $W(G)$ is finitely generated.
 	\end{theorem}
 
 The question whether $\X(G)$ is $FP_n$ when $G$ is $FP_n$ was resolved by Bridson and Kochloukova in \cite{BK} with afirmative answer for $n = 2$ and negative for $n \geq 3$.
 
 \begin{theorem} \cite{BK}  \label{fp2} If $G$ is finitey presented (resp. $FP_2$) then $\X(G)$ is finitely presented (resp. $FP_2$). But if $G$ is $FP_n$ for some $n \geq 3$  then $\X(G)$ is not necessary $FP_n$ since for $F$ a free non-cyclic group of finite rank $\X(F)$ is not $FP_3$.
 	\end{theorem} 
\section{The main results for $\Sigma^1(\X(G))$} \label{section-sigma1}

Throughout this section  $G$ is  a {\it finitely generated group},
 $$\chi : \X(G) \to \R$$ is a character and $$\chi_0 : \X(G) / D \to \R$$ is the character induced by $\chi$. 
 Note that $$\X(G) / D \simeq G \times \overline{G} \simeq G \times G.$$ We write $$\chi_0 = (\chi_1, \chi_2) : G \times G \to \R$$ and note that $\chi_1(g) = \chi(g)$ and $\chi_2(g) = \chi(\overline{g}_2)$.

\begin{lemma} \label{lemma1.1}Suppose that $G$ is a finitely generated group  and  $[\chi] \in \Sigma^1(\X(G))$. Then $[\chi_0] \in \Sigma^1(\X(G)/D)$ and  one of the following holds :
	 
	1. $\chi_1 \not= 0, \chi_2 \not= 0$;
	
	2. $ \chi_1 = 0, [\chi_2] \in \Sigma^1(G) $;  
	
	3. $ \chi_2 = 0, [\chi_1] \in \Sigma^1(G)$. 
	
	Furthermore if $\chi_1 = \chi_2 \not= 0$ then $[\chi_1] \in \Sigma^1(G)$.
 \end{lemma}

\begin{proof} The fact that  $[\chi_0] \in \Sigma^1(\X(G)/D)$   follows immediately by Lemma  
	\ref{quotient-sigma1}.  Note that $ \Sigma^1(\X(G)/D) \simeq G \times G$. By Theorem \ref{teo-conjec.prod.dir.corpo2} and the fact that $\Sigma^1 ( - , \Z) = \Sigma^1 ( - )$ we have
	\begin{equation} \label{sigma-prod} 
	\Sigma^1(G \times G) =
	\{ [(\chi_1, \chi_2)] \in S(G \times G) \mid \end{equation}  $$ \chi_1 \not= 0, \chi_2 \not= 0 \hbox{ or } \chi_1 = 0, [\chi_2] \in \Sigma^1(G)  \hbox{ or } \chi_2 = 0, [\chi_1] \in \Sigma^1(G) \}.$$
	
	If $\chi_1 = \chi_2$ then $\chi(L) = 0$ and $\chi$ induces a character $\widehat{\chi} : \X(G) / L \simeq G \to \R$. Note that $\widehat{\chi}$ can be identified with $\chi_1$ and by Lemma  
	\ref{quotient-sigma1} $[\widehat{\chi}] \in \Sigma^1(G)$.
	\end{proof}

 \begin{lemma} \label{lemma1.2} Suppose that  $G$ is a finitely generated group and $[\chi_0] =[ (\chi_1, \chi_2)] \in \Sigma^1(G \times G)$ and $\chi_1 \not= \chi_2$. Then $[\chi] \in \Sigma^1(\X(G))$.
 \end{lemma}
\begin{proof} 
	
	Note that since $\chi_1 \not= \chi_2$ we have $\chi(L) \not= 0$. From the very beginning we can fix an element $a \in L$ such that $\chi(a) \geq 1$ and include it in a fixed finite generating set $Y$ of $\X(G)$. Let $\Gamma$ be the Cayley graph of $\X(G)$ with respect to $Y$. Let $\widehat{Y}$ be the image of $Y$ in $\X(G) / D$ and let $\widehat{\Gamma}$ be the Cayley graph of $\X(G)/ D$ with respect to the finite generating set $\widehat{Y}$. By definition $\Gamma_{\chi}$ is the subgraph of $\Gamma$ spanned by $\X(G)_{\chi} = \{ h \in \X(G) ~| ~\chi(h) \geq 0 \}$ and $\widehat{\Gamma}_{\chi_0}$ is the subgraph of $\widehat{\Gamma}$ spanned by $(\X(G)/ D)_{\chi_0} = \{ h \in \X(G)/D ~| ~\chi_0(h) \geq 0 \}$.
	
Let $g \in \X(G)_{\chi} $ and write $\widehat{g}$ for the image of $g$ in $\X(G)/ D$. Since $[\chi_0] \in \Sigma^1(\X(G)/ D)$ we deduce that there is a path $\widehat{\gamma}$ in $\widehat{\Gamma}_{\chi_0}$ that starts at $1_{\X(G)/ D}$ and finishes at $\widehat{g}$. Then we can lift the path $\widehat{\gamma}$ to a path $\gamma$ in $\Gamma_{\chi}$ that starts at $1_{\X(G)}$ i.e. under the canonical epimorphism $\pi : F(Y) \to F(\widehat{Y})$ (where $F(Y)$ and $F(\widehat{Y})$ are the free groups with basis $Y$ and $\widehat{Y}$ respectively) the label $l(\gamma)$ of $\gamma$ is sent to the label $l(\widehat{\gamma})$. Then the path $\gamma$ finishes at an element of $\X(G)$ that is mapped under the canonical epimorphism $\X(G) \to \X(G) / D$ to $\widehat{g}$ i.e. the final point is $t g$ for some $t \in D$.
	
	Suppose there is a path $\gamma_0$ in $\Gamma_{\chi}$ that starts at $1_{\X(G)}$ and finishes at $t$. Then the composition path
	$\gamma_0^{-1} \gamma $ is a path in $\Gamma_{\chi}$ that starts at $t$ and finishes at $t g$. Finally since $\chi(t) \in \chi(D) = 0$ we deduce that $t^{-1} .(\gamma_0^{-1} \gamma)$ is a path in 
$\Gamma_{\chi}$ that starts at $1_{\X(G)}$ and finishes at $g$. Thus $[\chi] \in \Sigma^1(\X(G))$ as required.

Finally we construct the path $\gamma_0$. First we start with any path $\widetilde{\gamma}$ in $\Gamma$ that starts at $1_{\X(G)}$ and finishes at $t$. Note that for $m$ sufficiently large we have that $\widetilde{\gamma}$ is inside $\Gamma_{\chi \geq - m}$. Recall that $a \in L \cap Y$ is an element such that $\chi(a) \geq 1 $. Let $\delta_m$ be the path in $\Gamma_{\chi}$ that starts at $1_{\X(G)}$ and has label $a^m = a \ldots a$ and let $\delta_{-m}$ be the path in $\Gamma$ that starts at $1_{\X(G)}$ and has label $a^{-m} = a^{-1} \ldots a^{ -1}$.  Then the path $a^m .\widetilde{\gamma}$ is inside $\Gamma_{\chi}$ and starts at $a^m$ and finishes at $a^m t$. And since $\chi(t) = 0$ the path $a^m t . \delta_{-m}$ is inside $\Gamma_{\chi}$ and starts at $a^m t$ and finishes at $a^m t a^{-m}$. Note that $t \in D$ and $a \in L$. Since $[D, L] = 1$ we deduce that $a^m t a^{-m} = t$, hence
the concatenation $\gamma_0 = \delta_m (a^m . \widetilde{\gamma}) (a^m t. \delta_{-m})$ is a path inside $\Gamma_{\chi}$ that starts at $1_{\X(G)}$ and finishes at $t$.
\end{proof}

\begin{lemma} \label{lemma1.4} Suppose that $H$ is a finitely generated group. Let  $N$ be a finitely generated normal subgroup of $H$ and $\chi : H \to \mathbb{R}$ be a character such that $\chi(N) = 0$. Let $\widetilde{\chi} : H / N \to \mathbb{R}$ be the character induced by $\chi$. Assume that $[\widetilde{\chi}] \in \Sigma^1(H / N)$. Then $[\chi] \in \Sigma^1(H)$.

	In particular  for $H = \X(G)$, $N = L$ if $\chi_1 = \chi_2 \not= 0$ and $[\chi_1] \in \Sigma^1(G)$, then $[\chi] \in \Sigma^1(\X(G))$.
\end{lemma}

\begin{proof} 
	
	Let $\Gamma$ be the Cayley graph of $H$ with respect to a fixed finite generating set $Y$ and $\widehat{Y}$ be the image of $Y$ in $H / N$. Let $\widehat{\Gamma}$ be the Cayley graph of $H / N$ with respect to the generating set $\widehat{Y}$. 
	
	Fix $g \in H_{\chi}$ and consider $\widehat{g}$ the image of $g$ in $H/ N$.  Since $[\widetilde{\chi}] \in \Sigma^1(H/N)$ we deduce that there is a path $\widehat{\gamma}$ in $\widehat{\Gamma}_{\widetilde{\chi}}$ that starts at $1_{H/N}$ and finishes at $\widehat{g}$. Then we can lift the path $\widehat{\gamma}$ to a path $\gamma$ in $\Gamma_{\chi}$ that starts at $1_{H}$. Note that the path $\gamma$ finishes at an element of $H$ of the type $t g$ for some $t \in N$.
	
	Suppose there is a path $\gamma_0$ in $\Gamma_{\chi}$ that starts at $1_{H}$ and finishes at $t$. Then 
	$\gamma_0^{-1} \gamma $ is a path in $\Gamma_{\chi}$ with beginning $t$ and end $t g$. Finally since $\chi(t) \in \chi(N)  = 0$ we get $t^{-1} (\gamma_0^{-1} \gamma)$ is a path in 
	$\Gamma_{\chi}$ with beginning $1_{H}$ and end $g$. Thus $[\chi] \in \Sigma^1(H)$ as required.
	
	Finally we construct the path $\gamma_0$.   Consider a finite generating set $Y_1$ of $N$ and we can  choose $Y$ such that $Y_1 \subseteq Y$. Then we can link the elements $1_{H}$ and $t$ with a path $\gamma_0$ in $\Gamma_{\chi = 0}$ whose label is a word on $Y_1^{\pm 1}$, where $\Gamma_{\chi = 0}$ is the subgraph of $\Gamma$ generated by $Ker(\chi)$.

	Finally for the case $H = \X(G)$, $N = L$ observe that by \cite[Prop.~2.3]{BK} if $G$ is finitely generated then $L$ is finitely generated.
\end{proof}
 Lemma \ref{lemma1.1}, Lemma \ref{lemma1.2} and Lemma \ref{lemma1.4} imply the following corollary.
\begin{cor} \label{cor-sigma1}
	Let $G$ be a finitely generated group and $\chi : \X(G) \to \mathbb{R}$ be a character. Consider the homomorphisms $\chi_1, \chi_2 : G \to \R$, where $\chi_1(g) = \chi(g)$ and $\chi_2(g) = \chi(\overline{g})$.
	Then
	$[\chi] \in \Sigma^1(\X(G))$ if and only if one of the following holds :
	
	1. $\chi_1 \not= 0$ , $ \chi_2 \not= 0$ and $ \chi_1 \not= \chi_2$;
	
	2. $\chi_1 = 0 $, $ [\chi_2] \in \Sigma^1(G)$;
	
	3. $\chi_2 = 0 $, $ [\chi_1] \in \Sigma^1(G)$;
	
	4. $\chi_1 = \chi_2 \not= 0$ and $[\chi_1] \in \Sigma^1(G)$.
\end{cor}

Denote by $$\pi : \X(G) \to \X(G) / W$$ the canonical projection.
\begin{lemma}  \label{sigma1}  Suppose that $G$ is a finitely generated group. Let $$\widehat{\chi} : \X(G)/ W \to \R$$ be a character, $\chi_1, \chi_2 : G \to \mathbb{R}$ be the characters defined by  $\chi_1(g) = \widehat{\chi} \pi (g)$  and $\chi_2(g) = \widehat{\chi} \pi (\bar{g})$.
	
	Then $[\widehat{\chi}] \in
	\Sigma^1(\X(G)/ W)$ if and only one of the following conditions hold:

	1. $\chi_1 \not= 0$ , $ \chi_2 \not= 0$ and $ \chi_1 \not= \chi_2$;
	
	2. $\chi_1 = 0 $, $ [\chi_2] \in \Sigma^1(G)$;
	
	3. $\chi_2 = 0 $, $ [\chi_1] \in \Sigma^1(G)$;
	
	4. $\chi_1 = \chi_2 \not= 0$ and $[\chi_1] \in \Sigma^1(G)$.

\end{lemma} 
\begin{proof}
Let $$\chi = \widehat{\chi} \pi : \X(G) \to \mathbb{R}.$$ We claim that $[\widehat{\chi}] \in
\Sigma^1(\X(G)/ W)$ if and only if $[{\chi}] \in
\Sigma^1(\X(G))$. Indeed if $[{\chi}] \in
\Sigma^1(\X(G))$ by Lemma  
\ref{quotient-sigma1} $[\widehat{\chi}] \in
\Sigma^1(\X(G)/ W)$.

Suppose now that $[\widehat{\chi}] \in
\Sigma^1(\X(G)/ W)$. Then since $W \subseteq D$ the group $\X(G) / D$ is a quotient of $\X(G)/W$ and by Lemma  
\ref{quotient-sigma1} for the character $$(\chi_1, \chi_2) : \X(G)/ D \to \mathbb{R}$$ we have $[(\chi_1, \chi_2)] \in \Sigma^1(\X(G)/ D)$. By (\ref{sigma-prod}) either $\chi_1 = 0$, $[\chi_2] \in \Sigma^1(G)$ or $\chi_2 = 0$, $[\chi_1] \in \Sigma^1(G)$ or $\chi_1 \not= 0, \chi_2 \not= 0$. 

Suppose that $\chi_1 = \chi_2$. Note that for the epimorphism $\pi_0 : \X(G) / W \to \X(G) / L$ we have that $\widehat{\chi} = \pi_0 \chi_1$, where we have identified $\X(G) / L$ with $G$. Then by Lemma  
\ref{quotient-sigma1} since  $[\widehat{\chi}] \in
\Sigma^1(\X(G)/ W)$ we deduce that $[\chi_1] \in \Sigma^1(G)$. This completes the proof.

\end{proof}

{\bf Proof of Theorem A} It follows immediately from Corollary \ref{cor-sigma1} and Lemma \ref{sigma1}. 

\medskip
{\bf Proof of Corollary B2}
Suppose that $\pi_1(N)$, $\pi_2(N)$ and $\pi_3(N)$ are all finitely generated. Let $\chi : \X(G) \to \R$ be a character such that $\chi(N) = 0$. We aim to show that $[\chi] \in \Sigma^1(\X(G))$. Then by Theorem \ref{BRhomotopic} we will obtain that $N$ is finitely generated as required.

1. Suppose that $\chi_1 = 0$. Then $\chi_2 \not= 0$. Since $\chi(N) = 0$ we have $\chi_2(\pi_2(N)) = 0$. By Theorem \ref{BRhomotopic} the fact that $\pi_2(N)$ is finitely generated implies that $[\chi_2] \in \Sigma^1(G)$.

2. Suppose that $\chi_2 = 0$. Then $\chi_1 \not= 0$. Since $\chi(N) = 0$ we have $\chi_1(\pi_1(N)) = 0$. By Theorem \ref{BRhomotopic} the fact that $\pi_1(N)$ is finitely generated implies that $[\chi_1] \in \Sigma^1(G)$.

3. Suppose that $\chi_1 = \chi_2 \not= 0$. Since $\chi(N) = 0$ we have $\chi_1(\pi_3(N)) = 0$. By Theorem \ref{BRhomotopic} the fact that $\pi_3
(N)$ is finitely generated implies that $[\chi_1] \in \Sigma^1(G)$.

4. The final case is $\chi_1 \not= 0$, $\chi_2 \not= 0$ and $\chi_1 \not= \chi_2$.

 Then by Theorem A in all four cases $[\chi] \in \Sigma^1(\X(G))$ as required.

Finally apply the above for $N = \X(G)'$ to deduce that $\X(G)'$ is finitely generated if and only if $G'$ is finitely generated. 
This completes the proof of Corollary B2.

\medskip
{\bf Proof of Corollary C} Note that by Theorem \ref{limit0} for a limit group $G$ we have $\Sigma^1(G) = \emptyset$. Then by Theorem A 
$$
\Sigma^1(\X(G)) = \{ [\chi] \in S(\X(G)) \mid \chi_1 \not=0, \chi_2 \not= 0, \chi_1 \not= \chi_2 \}.$$

\section{Some results on $\Sigma^2(\X(G),\Z)$ and $\Sigma^2(\X(G))$}

In this section we prove results that do not require Theorem \ref{thmI}. Note that if $G$
 is $FP_2$ then by \cite{BK} $\X(G)$ is $FP_2$ too. The last condition is necessary for $\Sigma^2(\X(G), \Z) \not= \emptyset$ but as we will see from the results in this section it is not sufficient i.e. there are groups $G$ of type $FP_2$ such that $\Sigma^2(\X(G), \Z) = \emptyset$.

\begin{lemma} \label{abel123} Let $H$ be a group of type $FP_2$, $N$ a normal subgroup of $H$, $[\chi] \in \Sigma^2(H, \Z)$ such that $\chi(N) = 0$ and $\widetilde{\chi} : H/ N \to \R$ be the character induced by $\chi$. Suppose further that $N/[N,N]$ is finitely generated as a left $\Z (H/ N)_{\widetilde{\chi}}$-module, where $H/ N$ acts on $N/[N,N]$ via conjugation. Then
	$[\widetilde{\chi}] \in \Sigma^2(H/N, \Z)$.
	\end{lemma}

\begin{proof} Since $[\chi] \in \Sigma^2(H, \Z)$ there is a free resolution 
	$$
	{\mathcal P} : \ldots \to P_2 \to P_1 \to P_0 = \Z H_{\chi} \to \Z \to 0
	$$
	of the trivial left $\Z H_{\chi}$-module $\Z$, where $P_1$ and $P_2$ are   finitely generated as $\Z H_{\chi}$-modules. Consider the complex of free $\Z (H/ N)_{\widetilde{\chi}}$-modules
	$$
	{\mathcal R} = {\mathcal P} \otimes_{\Z N} \Z : \ldots \to R_2 \mapnew{d_2} R_1 \mapnew{d_1} R_0 = \Z (H/ N)_{\widetilde{\chi}} \to \Z \to 0
	$$
	Note that ${\mathcal R}$ is not in general exact and
	$$H_0({\mathcal R}) = 0 \hbox{ and } H_1({\mathcal R}) = H_1(N, \Z) = N/ [N,N],
	$$
	where the last follows from the fact that $\mathcal P$ can be viewed as a free resolution of $\Z N$-modules.
	Since $R_2$ is finitely generated as $\Z (H/ N)_{\widetilde{\chi}}$-module we conclude that $Im(d_2)$ is finitely generated as $\Z (H/ N)_{\widetilde{\chi}}$-module.
	This together with the fact that $$N/ [N,N] =  H_1({\mathcal R}) = Ker (d_1)/ Im (d_2)$$ is finitely generated as $\Z (H/ N)_{\widetilde{\chi}}$-module implies that $Ker(d_1)$ is finitely generated as $\Z (H/ N)_{\widetilde{\chi}}$-module. Hence $\Z$ is $FP_2$ as $\Z (H/ N)_{\widetilde{\chi}}$-module i.e. $[\widetilde{\chi}] \in \Sigma^2(H/ N, \Z)$.
	\end{proof}
	
	\begin{cor} \label{abel} Let $N$ be a normal subgroup of $\X(G)$, $[\chi] \in \Sigma^2(\X(G), \Z)$ such that $\chi(N) = 0$ and $\widetilde{\chi} : \X(G)/ N \to \R$ be the character induced by $\chi$. Suppose further that $N/[N,N]$ is finitely generated as a left $\Z (\X(G)/ N)_{\widetilde{\chi}}$-module, where $\X(G)/ N$ acts on $N/[N,N]$ via conjugation. Then
	$[\widetilde{\chi}] \in \Sigma^2(\X(G)/N, \Z)$.
	\end{cor}

\begin{prop} \label{Sigma2} Let $[\chi] \in \Sigma^2(\X(G), \Z)$ and $\chi_0 = (\chi_1, \chi_2): \X(G) / D \simeq G \times G \to \R$  and $\widehat{\chi} : \X(G) / W \to \R$ be the characters induced by $\chi$. Then the following conditions hold:
	
	1. if $\chi_1 \not=  \chi_2$ then $[\chi_0] \in \Sigma^2(\X(G)/ D, \Z)$;
	
	2. if $\chi_1 \not=  \chi_2$ then $[\widehat{\chi}] \in \Sigma^2(\X(G)/ W, \Z)$;
	
	3. if $\chi_1 = \chi_2$ then $\chi(L) = 0$ and for the character $\widetilde{\chi} : \X(G)/ L \to \R$ induced by $\chi$ we have $[\widetilde{\chi}] \in \Sigma^2( \X(G)/ L,\Z)$. Identifying $\X(G)/L$ with $G$ and $\widetilde{\chi}$ with $\chi_1$ we get $[\chi_1] \in \Sigma^2(G, \Z)$.
	
	\end{prop}

\begin{proof} Note that the condition $\chi_1 \not= \chi_2$ is equivalent to $\chi(L) \not= 0$.  Since $\Sigma^2(\X(G), \Z) \not= \emptyset$ we deduce that $\X(G)$ is $FP_2$, hence its retract $G$ is $FP_2$.
	
	1. By Corollary \ref{abel} applied for $N = D$
	it remains to prove that $D/ [D,D]$ is finitely generated as $\Z (\X(G)/ D)_{\chi_0}$-module. The fact that $G$ is $FP_2$ implies that 
	$\X(G) / D \simeq G \times G$ is $FP_2$ and so any relation module of $G \times G$ is finitely generated as $G \times G$-module. Hence any quotient of a relation module of $G \times G$
	is finitely generated as $G \times G$-module, in particular $D/[D,D]$ is finitely generated as $\X(G)$-module ( via conjugation). Since $\chi(L) \not= 0$ we have that $\X(G) = \X(G)_{\chi}  L$. This combined with the fact that $L$ and $D$ act trivially (via conjugation) on $D/ [D,D]$ implies that $D/[D,D]$ is finitely generated as $\Z ( \X(G)_{\chi}/ D)$-module. Finally note that $ \X(G)_{\chi}/ D = (\X(G)/D)_{{\chi}_0}$.
	
	2.  Note that $\X(G)/ W$ is a subdirect product of $G \times G \times G$ that maps surjectively on pairs. Thus since $G$ is $FP_2$ we can apply Theorem \ref{homological1-2-3} to deduce that $\X(G)/ W$ is $FP_2$, hence $W/ [W,W] = W$ is finitely generated as $\X(G)$-module via conjugation. Since $\chi(L) \not= 0$ we can use $\X(G) = \X(G)_{\chi} L$ and the fact that $L$ acts trivially on $W$ via conjugation to deduce 	that $W$ is finitely generated as $\Z ( \X(G)_{\chi}/ W)$-module. Finally note that $ \X(G)_{\chi}/ W = (\X(G)/W)_{\widehat{\chi}}$.
	
	3. Suppose now that $\chi(L) = 0$. This is equivalent to $\chi_1 = \chi_2$. 	
	Consider the decomposition $\X(G) = L \rtimes G$. Then the character $\chi$ induces a character $\widetilde{\chi} : \X(G) / L \to \R$ that after identifying $\X(G)/ L$ with $G$  is the character $\chi_1$. By Corollary \ref{abel}  to show that  $[\chi_1] \in \Sigma^2(G, \Z)$ it suffices to show that $L/ [L,L]$ is finitely generated as $\Z (\X(G)/ L)_{\widetilde{\chi}}$-module. As observed before Bridson and Kochloukova showed in \cite{BK} that $L$ is finitely generated whenever $G$ is finitely generated. The fact that $L/ [L,L]$ is finitely generated for finitely generated group $G$ was proved earlier by Lima and Oliveira  in \cite{LimaOliveira}. 
\end{proof}

\begin{lemma}  Let $H$ be a finitely presented group, $N$ a normal subgroup of $H$, $[\chi] \in \Sigma^2(H)$ such that $\chi(N) = 0$ and $\widetilde{\chi} : H/ N \to \R$ be the character induced by $\chi$. Suppose further that $N $ is finitely generated as a left $H_{\widetilde{\chi}}$-group, where $H$ acts on $N$ via conjugation. Then
	$[\widetilde{\chi}] \in \Sigma^2(H/N)$.
	\end{lemma}
	
	\begin{proof} 
	Since $[\chi] \in \Sigma^2(H)$ there is a finite presentation
	$H = \langle X ~| ~R \rangle$ such that for the Cayley complex associated to this presentation and its subcomplex $\Gamma_{\chi}$ spanned by the vertices $H_{\chi} = \{ h \in H ~| ~\chi(h) \geq 0 \}$ we have that
	$\Gamma_{\chi} $ is 1-connected. 
	The free left $H$-action on $\Gamma$  induces $N$-action on $\Gamma_{\chi}$ and thus we have a covering map
	$$p : \Gamma_{\chi} \to \Gamma_{\chi} / N$$
	Since $\pi_1(\Gamma_{\chi}) = 1$ we have that 
	$$N \simeq \pi_1(\Gamma_{\chi} / N).$$
	Since $N $ is finitely generated as a left $H_{\widetilde{\chi}}$-group, there are elements $b_1, \ldots, b_m \in N$ such that $N = \langle ^{ H_{\chi}} b_1, ~ \ldots, ~^{ H_{\chi}} b_m \rangle$.
	
	Consider the finite presentation $H/N = \langle X ~| ~R, b_1, \ldots, b_m \rangle$. The Cayley complex $\widehat{\Gamma}$ associated to this presentation is obtained from $\Gamma$ by gluing at each vertex extra 2-cells whose boundaries are closed paths $\gamma_1, \ldots, \gamma_m$ with labels that correspond to $b_1, \ldots, b_m$. Then there is a non-positive real number $d$ such  that $\gamma_1, \ldots, \gamma_m$ are closed  paths  homotopic to a point in $\widehat{\Gamma}_{\widetilde{\chi} \geq d}$, where $\widehat{\Gamma}_{\widetilde{\chi} \geq d}$ is the subcomplex of $\widehat{\Gamma}$ spanned by the vertices in $\{ g \in H/N ~| ~\widetilde{\chi}(g) \geq d \}$. Thus $\widehat{\Gamma}_{\widetilde{\chi} \geq 0}$ is $\widehat{\Gamma}_{\widetilde{\chi} }$. 
	
	The fact that 
	$\pi_1(\Gamma_{\chi} / N) \simeq N = \langle ^{ H_{\chi}} b_1, ~ \ldots, ~ ^{ H_{\chi}} b_m \rangle$ implies that the inclusion of spaces  $\widehat{\Gamma}_{\widetilde{\chi} } \subseteq \widehat{\Gamma}_{\widetilde{\chi} \geq d}$ induces  the trivial map
	$\pi_1(\widehat{\Gamma}_{\widetilde{\chi}}) \to  \pi_1(\widehat{\Gamma}_{\widetilde{\chi} \geq d}) $. This is one of the definitions of $\Sigma^2$, hence $\widetilde{\chi} \in \Sigma^2(H/ N)$.
	
	Alternatively we can assume from the very beginning that the fixed generating set $X$ contains a finite fixed subset of $H$. In particular we can assume that $X$ contains the set $\{ b_1, \ldots, b_m \}$. This guarantees that $\widehat{\Gamma}_{\widetilde{\chi}}$ is 1-connected.
	\end{proof}
	
	\begin{cor} \label{abel-top} Let $N$ be a normal subgroup of $\X(G)$, $[\chi] \in \Sigma^2(\X(G))$ such that $\chi(N) = 0$ and $\widetilde{\chi} : \X(G)/ N \to \R$ be the character induced by $\chi$. Suppose further that $N$ is finitely generated as a left $\X(G)_{\widetilde{\chi}}$-group, where $\X(G)$ acts (on the left) on $N$ via conjugation. Then
	$[\widetilde{\chi}] \in \Sigma^2(\X(G)/N)$.
	\end{cor}

\begin{prop} \label{Sigma2top} Let $G$ be a finitely presented group, $[\chi] \in \Sigma^2(\X(G))$, $\chi_0 = (\chi_1, \chi_2): \X(G) / D \simeq G \times G \to \R$  and $\widehat{\chi} : \X(G) / W \to \R$ be the characters induced by $\chi$. Then the following conditions hold:
	
	1. if $\chi_1 \not=  \chi_2$ then $[\chi_0] \in \Sigma^2(\X(G)/ D)$;
	
	2. if $\chi_1 \not=  \chi_2$ then $[\widehat{\chi}] \in \Sigma^2(\X(G)/ W)$;
	
	3. if $\chi_1 = \chi_2$ then $\chi(L) = 0$ and for the character $\widetilde{\chi} : \X(G)/ L \to \R$ induced by $\chi$ we have $[\widetilde{\chi}] \in \Sigma^2( \X(G)/ L)$. Identifying $\X(G)/L$ with $G$ and $\widetilde{\chi}$ with $\chi_1$ we get $[\chi_1] \in \Sigma^2(G)$.
	
	\end{prop}

\begin{proof} By Theorem \ref{fp2} since $G$ is finitely presented, $\X(G)$ is finitely presented.
	
	1. By Corollary \ref{abel} applied for $N = D$
	it remains to prove that $D$ is finitely generated as $ \X(G)_{\chi_0}$-group where $ \X(G)_{\chi_0}$ acts (on the left) via conjugation. The fact that $G$ is finitely presented implies that 
	$\X(G) / D \simeq G \times G$ is finitely presented. Hence $D$ is finitely generated as a normal subgroup of $\X(G)$ i.e. is finitely generated as $ \X(G)$-group where $ \X(G)$ acts (on the left) via conjugation.Since $\chi(L) \not= 0$ we have that $\X(G) =  \X(G)_{\chi} L $. This combined with the fact that $[L, D] = 1$ implies that $L$ acts trivially on $D$ via conjugation, hence the $\X(G)$ action on $D$ via conjugation factors through an action of $\X(G)_{\chi}$. 
	
	2.  Note that $\X(G)/ W$ is a subdirect product of $G \times G \times G$ that maps surjectively on pairs. Since by \cite{BHMS}
	the Virtual Surjection Conjecture holds for $n = 2$,  we deduce that $\X(G)/ W$ is finitely presented, hence $W$ is finitely generated as a normal subgroup of $\X(G)$. Since $W$ is abelian this is equivalent to $W$ is finitely generated as  a left $\mathbb{Z} \X(G)$-module via conjugation. Since $\chi(L) \not= 0$ we can use $\X(G) = \X(G)_{\chi} L $ and the fact that $L$ acts trivially on $W$ via conjugation (since $W = L \cap D$ and $[L,D] = 1$), to deduce 	that $W$ is finitely generated as a left $\Z  \X(G)_{\chi}$-module. 
	
	3. Suppose now that $\chi(L) = 0$. This is equivalent to $\chi_1 = \chi_2$. 	
	Consider the decomposition $\X(G) = L \rtimes G$. Then the character $\chi$ induces a character $\widetilde{\chi} : \X(G) / L \to \R$ that after identifying $\X(G)/ L$ with $G$  is the character $\chi_1$. By Corollary \ref{abel-top}  to show that  $[\chi_1] \in \Sigma^2(G)$ it suffices to show that $L$ is finitely generated as a left $\X(G)_{\widetilde{\chi}}$-group. As observed before Bridson and Kochloukova showed in \cite{BK} that $L$ is finitely generated  as a group whenever $G$ is finitely generated.
\end{proof}

{\bf Proof of Proposition D}
Suppose that $[\chi] \in \Sigma^2(\X(G), \Z)$. Then by Proposition \ref{Sigma2} for the induced character $$\chi_0 = ( \chi_1, \chi_2) : \X(G) / D \simeq G \times G \to \R$$
	either $[\chi_0] \in \Sigma^2(G \times G, \Z)$ or $\chi_1 = \chi_2$ and $[\chi_1] \in \Sigma^2(G, \Z)$. Since $\Sigma^2(G, \Z) \subseteq \Sigma^1(G, \Z) = \Sigma^1(G)$ we have that $\Sigma^2(G, \Z)$ is empty  if $\Sigma^1(G)$ is empty. Note that by Theorem \ref{teo-conjec.prod.dir.corpo2} (i.e. the direct product formula holds in dimension 2), hence    $\Sigma^2(G \times G, \Z)$ is empty  if $\Sigma^1(G)$ is empty. 
This completes the proof of Proposition D.
	
	\medskip
	
	\begin{theorem} \cite[Cor.~4.2]{Me} \label{homo123} Let $H$ be a finitely presented group, $N$ a normal subgroup of $H$ that is finitely presented, $\mu : H/N \to \mathbb{R}$ a character and $\pi: H \to H/N$ the canonical epimorphism. Then $[\mu] \in \Sigma^2(H/N)$ if and only if $[\mu \pi] \in \Sigma^2(H)$.
	\end{theorem}
	We will need the homological version of the above result.
	
	\begin{theorem} \label{Npres} 
	Suppose $N$ is a normal subgroup of $H$ such that both $N$ and $H$ are $FP_2$, $\chi : H \to \R$ is a character such that $\chi(N) = 0$, $\widetilde{\chi} : H / N \to \R$ is the character induced by $\chi$ and $[\widetilde{\chi}] \in \Sigma^2(H/N, \Z)$. Then $[\chi] \in \Sigma^2(H, \Z)$.
\end{theorem}

\begin{proof}
	Consider the short exact sequence of groups
	$$1 \to N \to H \to H/N \to 1$$ and the induced short exact sequence of
	monoids
	$$1 \to N \to H_{\chi} \to (H/N)_{\widetilde{\chi}} \to 1$$
	This induces a LHS spectral sequence
	$$ E_{p,q}^2 = H_p ( (H/N)_{\widetilde{\chi}}, H_q(N, V)) = Tor_p^{\Z (H/N)_{\widetilde{\chi}}}(\Z, Tor_q^{\Z N}(\Z,V))$$ that converges to $H_{p+q} (H_{\chi}, V) = Tor_{p+q}^{\Z H_{\chi}}(\Z, V)$ for a fixed $\Z H_{\chi}$-module $V$. 
	We set $V = \prod \Z H_{\chi}$.
	
	Note that $[\widetilde{\chi}] \in \Sigma^2( H/N, \Z)$ is equivalent to $Tor_i^{(H/N)_{\widetilde{\chi}}}(\Z, -)$ commutes with direct products for $i = 0$ and $i = 1$. Note that this follows from Bieri's criterion \cite[Thm.~1.3, iiia)', p. ~12]{Bieribook}.
	By the same argument $[\chi] \in \Sigma^2(H, \Z)$ is equivalent  to  $Tor_i^{\Z H_{\chi}}(\Z, \prod \Z H_{\chi}) = 0$ for $i = 1$ and $Tor_0^{\Z H_{\chi}}(\Z, \prod \Z H_{\chi}) \simeq \prod Tor_0^{\Z H_{\chi}}(\Z, \Z H_{\chi})$ . 
	 
	1)  Note that
	$$
	E_{1,0}^2 =  Tor_1^{\Z (H/N)_{\widetilde{\chi}}}(\Z, Tor_0^{\Z N}(\Z, \prod \Z H_{\chi}  ))
	$$
	Since $N$ is finitely generated, we deduce that $Tor_0^{\Z N}(\Z, -)$ commutes with direct products. Thus 
	\begin{equation} \label{eqi}  H_0(N,   \prod \Z H_{\chi}  ) = 
	Tor_0^{\Z N}(\Z, \prod \Z H_{\chi}  ) \simeq \prod Tor_0^{\Z N}(\Z,  \Z H_{\chi}  ) \simeq \end{equation} $$ \prod  (\Z H_{\chi} / N) \simeq \prod \Z (H/N)_{\widetilde{\chi}}.$$
	Since $[\widetilde{\chi}] \in \Sigma^2((H/N)_{\widetilde{\chi}}, \Z)$ we deduce that
	$$Tor_1^{\Z (H/N)_{\widetilde{\chi}}}(\Z, \prod \Z (H/N)_{\widetilde{\chi}}) \simeq 
	\prod Tor_1^{\Z (H/N)_{\widetilde{\chi}}}(\Z, \Z (H/N)_{\widetilde{\chi}})  = \prod 0 = 0.
	$$
	Combining the above equalities we deduce that 
	$$	E_{1,0}^2 = 0, \hbox{ hence 	}E_{1,0}^{\infty} = 0.$$
	Now 
	$$
	E_{0,1}^2 =  Tor_0^{\Z (H/N)_{\widetilde{\chi}}}(\Z, Tor_1^{\Z N}(\Z, \prod \Z H_{\chi}  )).
	$$
	Note that $N$ is $FP_2$, hence $Tor_1^{\Z N}(\Z, - )$ commutes with direct products. Thus $$Tor_1^{\Z N}(\Z, \prod \Z H_{\chi}  ) \simeq \prod Tor_1^{\Z N}(\Z,  \Z H_{\chi}  ) \simeq \prod 0 = 0, $$
	where we have used that $ \Z H_{\chi} $ is a free $\Z N$-module and so $Tor_1^{\Z N}(\Z,  \Z H_{\chi}  ) = 0$.  Hence
	$$	E_{0,1}^2 =  Tor_0^{\Z (H/N)_{\widetilde{\chi}}}(\Z, 0) = 0.$$
	Then $E_{0,1}^{\infty} = E_{1,0}^{\infty} = 0$. Finally the convergence of the spectral sequence gives a short exact sequence of abelian groups
	$$ 0 \to E_{0,1}^{\infty} \to H_1 (H_{\chi}, \prod \Z \X(G)_{\chi}) \to E_{1,0}^{\infty} \to 0,
	$$
	hence
	$$H_1 (H_{\chi}, \prod \Z H_{\chi})  = 0.$$  
	
	2) It remains to show that $Tor_0^{\Z H_{\chi}}(\Z, \prod \Z H_{\chi}) \simeq \prod Tor_0^{\Z H_{\chi}}(\Z, \Z H_{\chi})$. This is equivalent to $[\chi] \in \Sigma^1 (H)$ and follows from Lemma \ref{lemma1.4}. 
\end{proof}

\begin{cor} \label{Npres2} 
	Suppose $N$ is a normal subgroup of $\X(G)$ such that $N$ and $G$ are  $FP_2$, $\chi : \X(G) \to \R$ is a character such that $\chi(N) = 0$, $\widetilde{\chi} : \X(G) / N \to \R$ is the character induced by $\chi$ and $[\widetilde{\chi}] \in \Sigma^2(\X(G)/N, \Z)$. Then $[\chi] \in \Sigma^2(\X(G), \Z)$.
\end{cor}

\begin{cor} \label{Lpres} 
	Suppose $\chi : \X(G) \to \R$ is a character such that $\chi(L) = 0$ and $\chi_1 = \chi \mid_G$. Suppose that $L$ is $FP_2$ and that $[{\chi}_1] \in \Sigma^2(G, \Z)$. Then $[\chi] \in \Sigma^2(\X(G), \Z)$.
\end{cor}

\begin{proof} Since $\Sigma^2(G, \Z) \not= \emptyset$ we have that $G$ is $FP_2$, hence by Theorem \ref{fp2} $\X(G)$ is $FP_2$. We apply Corollary \ref{Npres2} for $N = L$ and identify $\X(G)/ L$ with $G$. Note that the character $\widehat{\chi} : \X(G)/ L \to \R$ under  the identification of $\X(G)/L$ with $G$ is identified with $\chi_1$.
	\end{proof}

\section{Proofs of Theorem E1, Theorem E2, Theorem F1, Theorem F2, Corollary G and Corollary H}

Let $H = G \times G \times G$ and $K = Im (\rho)$. Recall that $$Im (\rho) = \{ (g_1, g_2, g_3) \mid g_1 g_2^{-1} g_3 \in G' \},$$ hence
$$[Im (\rho), Im (\rho)] = [H, H].$$

{\bf Proof of Theorem E1}

	Suppose that 
$$\mu = (\mu_1, \mu_2, \mu_3)  : H = G \times G \times G \to \R$$
be a character extending $\chi$. 
By Theorem \ref{thmI} $[\chi] \in \Sigma^2(K,\Z)$ if and only if for every character $\mu$ extending $\chi$ we have $[\mu] \in \Sigma^2(H, \Z)$. Note that in dimension 2 the $\Sigma$ direct product formula holds, see Theorem \ref{teo-conjec.prod.dir.corpo2}, hence $[\mu] \in \Sigma^2(H, \Z)$ precisely if one of the following cases holds for the characters $\mu_1, \mu_2, \mu_3$ :

a) Two characters from $\{ \mu_1, \mu_2, \mu_3 \}$ are 0 and the third corresponds to an element of $\Sigma^2(G, \Z)$;

b) One character from $\{ \mu_1, \mu_2, \mu_3 \}$ is 0, and the other two are non-zero and at least one corresponds to an element of $\Sigma^1(G)$.

c) The three characters $\mu_1, \mu_2, \mu_3$ are non-zero.

Note that since $\mu$ is an extension of $\chi$ we have
$$
\chi((g_1, g_1 g_3, g_3)) = \mu_1(g_1) + \mu_2(g_1 g_3) + \mu_3(g_3).
$$
Hence
\begin{equation} \label{praiagrande1} \chi_1(g_1) = \chi( ( g_1, g_1, 1)) = \mu_1(g_1) + \mu_2(g_1), \hbox{ i.e. } \  \chi_1 = \mu_1 + \mu_2 \end{equation}
and
\begin{equation} \label{praiagrande2} 
\chi_2(g_3) = \chi( (1, g_3, g_3)) = \mu_2(g_3) + \mu_3(g_3) \hbox{ i.e. } \chi_2 = \mu_2 + \mu_3.
\end{equation}

Then $\mu_1 = \chi_1 - \mu_2$ and $\mu_3 = \chi_2 - \mu_2$. Suppose that for each $\mu$ that extends $\chi$ we have   $[\mu] \in \Sigma^2(H, \Z)$.  If $\mu_1 \not=0, \mu_2 \not= 0, \mu_3 \not= 0$ we are in case c) and $[\mu] \in \Sigma^2(H,\Z)$. There are several more cases to consider.

a1) Assume that $\mu_1 = 0 = \mu_2$. Then by (\ref{praiagrande1}) and (\ref{praiagrande2}) $\mu_3 = \chi_2$, $\chi_1 = 0$ and $[\mu] \in \Sigma^2(H,\Z)$  is equivalent to $[\mu_3] \in \Sigma^2(G,\Z)$. Thus 
$$\chi_1 = 0,  [\chi_2] \in \Sigma^2(G, \Z).$$

a2) Assume that $\mu_1 = 0 = \mu_3$. Then by (\ref{praiagrande1}) and (\ref{praiagrande2})  $\chi_1 = \mu_2 = \chi_2$ and  $[\mu] \in \Sigma^2(H,\Z)$  is equivalent to 
 $[\mu_2] \in \Sigma^2(G,\Z)$. Thus  
$$\chi_1 = \chi_2, [\chi_1] \in \Sigma^2(G, \Z).$$

a3) Assume that $\mu_2 = 0 = \mu_3$. Then by (\ref{praiagrande1}) and  (\ref{praiagrande2}) $\chi_2 = 0$, $\mu_1 = \chi_1$ and   $[\mu] \in \Sigma^2(H, \Z)$  is equivalent to  $[\mu_1] \in \Sigma^2(G,\Z)$. Thus
$$\chi_2 = 0,  [\chi_1] \in \Sigma^2(G, \Z).$$

b1) Assume that $\mu_1 = 0, \mu_2 \not=0, \mu_3 \not= 0$. Then by (\ref{praiagrande1}) and (\ref{praiagrande2}) $\mu_2 = \chi_1$, $\mu_3 = \chi_2 - \mu_2 = \chi_2 - \chi_1$.
Then $[\mu] \in \Sigma^2(H, \Z)$  is equivalent to both $\mu_2$ and $\mu_3$ are non-zero and at least one represents an element from $\Sigma^1(G)$ i.e.
$$ \chi_1 \not= 0, \chi_2 \not= \chi_1 \hbox{ and at least one of the elements of } \{ [\chi_1], [\chi_2 - \chi_1] \} \hbox{ belongs to } \Sigma^1(G).$$

b2) Assume that $\mu_3 = 0, \mu_2 \not=0, \mu_1 \not= 0$. Then by (\ref{praiagrande1}) and (\ref{praiagrande2}) $\mu_2 = \chi_2$, $\mu_1 = \chi_1 - \mu_2 = \chi_1 - \chi_2$. 
Then $[\mu] \in \Sigma^2(H, \Z)$  is equivalent to both $\mu_1$ and $\mu_2$ are non-zero and at least one represents an element from $\Sigma^1(G)$ i.e.
$$  \chi_2 \not= 0, \chi_1 \not= \chi_2 \hbox{ and at least one of the elements of } \{ [\chi_2], [\chi_1 - \chi_2] \} \hbox{ belongs to } \Sigma^1(G).$$

b3) Assume that $\mu_2 = 0$, $  \mu_1 \not= 0, \mu_3 \not= 0$.
Then by (\ref{praiagrande1}) and (\ref{praiagrande2}) $\mu_1 = \chi_1$, $\mu_3 = \chi_2$. Then $[\mu] \in \Sigma^2(H, \Z)$  is equivalent to both $\mu_1$ and $\mu_3$ are non-zero and at least one represents an element from $\Sigma^1(G)$ i.e.
$$\chi_2 \not= 0, \chi_2 \not= \chi_1 \hbox{ and at least one of the elements of }  \{ [\chi_1], [\chi_2] \} \hbox{ belongs to } \Sigma^1(G).$$

As a corollary of b1), b2) and b3) 
 if $\chi_1 \not= 0, \chi_2 \not= 0$ and $\chi_1 \not= \chi_2$ one of the following conditions should hold
 $$
 \{  [\chi_1], [\chi_2]  \} \subseteq \Sigma^1(G) \hbox{ or } \{  [\chi_1], [\chi_1 - \chi_2]  \} \subseteq \Sigma^1(G) 
 \hbox{ or } \{  [\chi_2], [\chi_2 - \chi_1]  \} \subseteq \Sigma^1(G) \
 $$ This completes the proof of Theorem E1.
 
 \medskip
 
 We note that the proof of Theorem E2 is similar to the proof of Theorem E1, since we can use the homotopical part of Theorem \ref{thmI}.
 
 \medskip

 	\medskip
{\bf Proof of Theorem F1}

a)	Assume now that $[\chi] \in \Sigma^2(\X(G), \Z)$.  By part 2, Proposition \ref{Sigma2}  if $\chi(L) \not= 0$, i.e. $\chi_1 \not= \chi_2$, we deduce that 	$[\widehat{\chi}] \in \Sigma^2(\X(G)/ W, \Z)$.  

By Proposition \ref{Sigma2} when $\chi_1 = \chi_2$ we have $[\chi_1] \in \Sigma^2(G, \mathbb{Z})$. Then by Theorem E1 $[\widehat{\chi}] \in \Sigma^2(\X(G)/ W, \Z)$.

b)  	Note that by  Theorem \ref{commutator} if $G$ is $FP_2$ and the abelianization of the commutator group $[G,G]$ is finitely generated then  $W(G)$ is finitely generated as abelian group.

Note that by Corollary \ref{Npres2}  for $N = W$  when $W$ is finitely presented (in our case  it is a finitely generated abelian group) and
$[\widehat{\chi}] \in \Sigma^2(\X(G)/ W, \Z)$  we can deduce that $[\chi] \in \Sigma^2(\X(G), \Z)$.

\medskip
{\bf Proof of Theorem F2} a)	Assume now that $[\chi] \in \Sigma^2(\X(G))$.   By part 2, Proposition \ref{Sigma2top} if $\chi(L) \not= 0$, i.e. $\chi_1 \not= \chi_2$, we deduce that 	$[\widehat{\chi}] \in \Sigma^2(\X(G)/ W)$.  

By Proposition \ref{Sigma2top} when $\chi_1 = \chi_2$ we have $[\chi_1] \in \Sigma^2(G)$. Then by Theorem E2 $[\widehat{\chi}] \in \Sigma^2(\X(G)/ W)$.

b)  	
Note that by Corollary  \ref{Npres2}  for $N = W$  when $W$ is finitely presented and
$[\widehat{\chi}] \in \Sigma^2(\X(G)/ W)$  we can deduce that $[\chi] \in \Sigma^2(\X(G))$.

\medskip
{\bf Proof of Corollary G} a) By Theorem A $[\chi] \in \Sigma^1(\X(G))^c$ if and only if one of the following conditions holds:

1) $\chi_2 = 0, [\chi_1] \in \Sigma^1(G) ^c$;

2) $\chi_1 = 0, [\chi_2] \in \Sigma^1(G) ^c$;

3) $\chi_1 = \chi_2 \not= 0$ and $[\chi_1] \in \Sigma^1(G)^c$,

where $\chi_1, \chi_2 : G \to \mathbb{R}$ are characters defined by $\chi_1(g) = \chi(g)$ and $\chi_2(g) = \chi( \overline{g})$.

By the definition of $V_i$ we have that $[\chi] \in V_i$ if and only if $\chi$ satisfies the $i$-th condition. Hence
$$\Sigma^1(G)^c = V_1 \cup V_2 \cup V_3$$

b) Identifying $S(\X(G))$ with $S(\X(G)/ W(G))$ via the projection map $\X(G) \to \X(G)/ W(G)$ and by Theorem F2 we have 
$\Sigma^2(\X(G)) \subseteq \Sigma^2(\X(G)/ W(G))$. Hence 
$$\Sigma^2(\X(G)/ W(G))^c \subseteq \Sigma^2(\X(G))^c$$
By Theorem E2
$[\chi] \in \Sigma^2(\X(G)/W(G))^c$ if and only if on of the following conditions holds:

1) $\chi_2 = 0, [\chi_1] \in \Sigma^2(G) ^c$;

2) $\chi_1 = 0, [\chi_2] \in \Sigma^2(G) ^c$;

3) $\chi_1 = \chi_2 \not= 0$ and $[\chi_1] \in \Sigma^2(G)^c$;

4)$\chi_1 \not= 0, \chi_2 \not= 0$, $\chi_1 \not= \chi_2$ and one of the following holds :

4a) $\{ [\chi_1], [\chi_2] \} \subseteq \Sigma^1(G)^c$;

4b) $\{ [\chi_1], [\chi_2 - \chi_1] \} \subseteq \Sigma^1(G)^c$;

4c) $\{ [\chi_2], [\chi_1 - \chi_2] \} \subseteq \Sigma^1(G)^c$;

By the definition of $W_i$ we have that $[\chi] \in W_i$ if and only if $\chi$ satisfies the $i$-th condition for $1 \leq i \leq 3$.

In the case 4a) $[\chi] = [( \chi_1, \chi_2)] =  [(\chi_1, 0)] + [(0, \chi_2)] $ is a typical element of  $V_1 + V_2$.

In the case 4b) $[\chi] = [( \chi_1, \chi_2)] =[(0, \chi_2 - \chi_1)] + [(\chi_1, \chi_1)]  $ is a typical element of  $V_2 + V_3$. 

In the case 4c) $[\chi] = [( \chi_1, \chi_2)] =  [(\chi_2, \chi_2)] + [(\chi_1 - \chi_2, 0)]  $ is a typical element of  $V_3 + V_1$.

c) The proof is the obvious homological modification of b).

\medskip
{\bf Proof of Corollary H}
 	Let  $\mu : \X(G) / W(G) \to \R$ be a character that vanishes on $N/W(G)$. We define a character $\chi : \X(G) \to \R$ as the composition of the canonical projection $\X(G) \to \X(G)/ W(G)$ with  $\mu$. Thus $\mu = \widehat{\chi}$ and since $\chi(N) = 0$ and $N$ is $FP_2$ (resp. finitely presented) by Theorem \ref{BRhomotopic} $[\chi] \in \Sigma^2(\X(G), \Z)$ (resp.  $[\chi] \in \Sigma^2(\X(G))$). Then by Theorem F1 $[\widehat{\chi}] \in \Sigma^2(\X(G)/ W(G), \Z)$ (resp. by Theorem F2 $[\widehat{\chi}] \in \Sigma^2(\X(G)/ W(G))$ ), hence by Theorem \ref{BRhomotopic} again $N/ W(G)$ is $FP_2$ (resp. finitely presented). 
 	
 	Consider the case $N = \X(G)'$. Note that $\X(G) = L \rtimes G$, hence $\X(G)' = \langle L', [L,G] \rangle \rtimes G'$. Suppose that $\X(G)'$ is $FP_m$ for some $m \geq 1$ (resp. finitely presented). Since property $FP_m$    (resp. finitely presented) passes to retracts we deduce that $G'$ is $FP_m$  (resp. finitely presented). 
 	
 	For the converse assume that $G'$ is $FP_m$ for some $m \geq 1$  (resp. finitely presented). Since $FP_1$ is equivalente with finite generation, we have that $G'$ is finitely generated  and by Theorem \ref{commutator} $W(G)$ is finitely generated. Then  $$\X(G)' / W = Im (\rho)' = G' \times G' \times G' \hbox{  is }FP_m \hbox{ (resp. finitely presented). }$$ Since $W$ is abelian and finitely generated, it is of type $FP_m$ and finitely presented. This implies that $\X(G)'$ is $FP_m$ (resp. finitely presented) as claimed and completes the proof of Corollary H.
 	
 	\medskip
 	{\bf Remark} Though it is tempting to study the structure of $\Sigma^n(\X(G), \Z)$  and $\Sigma^n(\X(G))$ for $n \geq 3$ there are some structural problems. Firstly, we do not have a criterion when $\X(G)$ is of type $FP_3$ and by \cite{BK} even for nice groups such as  finite rank free non-cyclic groups $G$ the group $\X(G)$ is not $FP_3$. Secondly, if we want to find a higher dimensional version of Theorem E1 and Theorem E2, it is natural to  apply Theorem \ref{thmI} for $H = G \times G \times G$ and $K = Im (\rho) \simeq \X(G)/ W(G)$, where the map $\rho : \X(G) \to H$  will be discussed in the preliminary section \ref{prel-X(G)}.  But though there is a direct product $\Sigma$-formula that holds for the homological $\Sigma$-invariants with coefficients in a field \cite{B-G} a similar formula does not hold for $\Sigma^n ( -, \mathbb{Z})$ for $n \geq 4$ \cite{Schutz} or for $\Sigma^n( -)$  for $n \geq 3$ \cite{Meinert-VanWyk}.
 	
	\section{On the finite presentability of the non-abelian tensor square and on the $\Sigma$-invariants of $\nu(G)$} \label{section-tensor}

	Let $G$ be a group. In \cite{Rocco}  Rocco defined a group  given by the following presentation
$$\nu(G) = \langle G, \overline{G} \mid [g_1, \overline{g}_2]^{g_3} = [ g_1^{g_3}, \overline{g_2^{g_3}}] = [g_1, \overline{g}_2]^{\overline{g}_3}\rangle, $$
where $\overline{G}$ is an isomorphic copy of $G$. By \cite{Rocco} for the non-abelian tensor square $G \otimes G$ we have an isomorphism $$G \otimes G \simeq [G, \overline{G}],$$
where $[G, \overline{G}]$ is the subgroup of $\nu(G)$ generated by $\{  [g_1, \overline{g}_2] ~| g_1 \in G, \overline{g}_2 \in \overline{G} \}$.
Furthermore by \cite{Rocco} there is a subgroup $\Delta \subseteq \nu(G)' \cap Z(\nu(G))$ such that
\begin{equation} \label{equ0} \nu(G)/ \Delta \simeq \X(G)/ R,\end{equation}
with an isomorphism that is an identity on $G \cup \overline{G}$,
$R$ is a special normal subgroup of $\X(G)$ that is contained in $W = W(G)$ and $W/ R \simeq H_2(G, \mathbb{Z})$. Thus
$\Delta$ is a quotient of $H_2(\nu(G)/ \Delta, \mathbb{Z}) \simeq 
H_2(\X(G)/ R, \mathbb{Z})$.

\begin{lemma} \label{neueW} Let $G$ be a group of type $FP_2$. Then $\Delta$ is finitely generated and there is a normal subgroup $W_0$ in $\nu(G)$ such that $W_0$ is finitely generated nilpotent of class at most 2 and 
\begin{equation} \label{iso321} \nu(G)/ W_0 \simeq \X(G)/ W. \end{equation}
\end{lemma}

\begin{proof} Let $W_0$ be the normal subgroup of $\nu(G)$ such that
$W_0 / \Delta$ is the preimage of $W/ R$ in $\nu(G) / \Delta$ under the isomorphism (\ref{equ0}).
Then $$W_0 / \Delta \simeq W/ R \simeq H_2(G, \mathbb{Z})$$
and $\Delta$ is a quotient of $
H_2(\X(G)/ R, \mathbb{Z})$. Since $G$ is $FP_2$ then $H_2(G, \mathbb{Z})$ is finitely generated, hence $W / R \simeq H_2(G, \mathbb{Z})$ is a finitely generated abelian group, hence is finitely presented and so is $FP_2$. Furthermore when $G$ is $FP_2$ by \cite[Thm. ~ D]{KochSidki} $\X(G)/ W$ is $FP_2$. Since the property $FP_2$ is extension closed \cite[Exer., ~p.~23] {Bieribook} we deduce that $\X(G)/ R$ is $FP_2$, hence $H_2(\X(G)/ R, \mathbb{Z})$ is finitely generated. Then  its quotient $\Delta$ is a finitely generated central subgroup of $\nu(G)$. Then $W_0$ is nilpotent of class at most 2 and  both $\Delta$ and $W_0/ \Delta \simeq H_2(G, \mathbb{Z})$ are finitely generated. Then we can deduce that $W_0$ is finitely generated.
\end{proof}

\medskip
{\bf Proof of Proposition J} 
Let $W_0$ be the normal subgroup of $\nu(G)$ given by Lemma \ref{neueW}. 
 By construction
$$\nu(G)/ W_0 \simeq \X(G)/ W \simeq Im (\rho).$$
Since $W \subseteq [G, \overline{G}] \subseteq \X(G)$ we have that $W_0 \subseteq [G, \overline{G}] \subseteq \nu(G)$, hence
$$[G, \overline{G}]/ W_0 \simeq [\rho(G), \rho(\overline{G})] \subseteq Im (\rho)  \simeq \X(G)/ W.$$ 
Since $[\rho(G), \rho(\overline{G})] = 1 \times G'\times 1 \simeq G '$ is finitely presented (resp. $FP_2$), $[G, \overline{G}]/ W_0$ is finitely presented (resp. $FP_2$).
 Combining with the fact that $W_0$ is finitely presented, we conclude that 
$$G \otimes G \simeq [G, \overline{G}] $$
is finitely presented (resp. $FP_2$). This completes the proof of Proposition J.

\begin{prop} Let $G$ be a group. We identify $S(\nu(G))$ with $S(\nu(G)/ W_0)$ via the canonical projection map $\nu(G) \to \nu(G)/ W_0$ and we identify $S(\nu(G)/ W_0)$ with $S(\X(G)/ W)$ via the isomorphism (\ref{iso321}). Then

a)  $\Sigma^2(\nu(G)) = \Sigma^2(\X(G)/ W)$ if $G$ is finitely presented;

b) $\Sigma^2(\nu(G), \mathbb{Z}) = \Sigma^2(\X(G)/ W, \mathbb{Z})$  
 if $G$ is $FP_2$;

{Remark} We recall that in Theorem E1 and Theorem E2 we have calculated  both $\Sigma^2(\X(G)/ W, \mathbb{Z})$ and $\Sigma^2(\X(G)/ W)$.
\end{prop}

\begin{proof}
a), b) In both cases $G$ is $FP_2$ and by Lemma \ref{neueW}. $W_0$ is finitely generated nilpotent of class at most 2. Hence $W_0$ is finitely presented. By construction $W_0 \subseteq \nu(G)'$ and identifying $\nu(G)/ W_0$ with $\X(G)/ W$ we have 
$\Sigma^2(\nu(G)/ W_0) = \Sigma^2(\X(G)/ W)$ and
$\Sigma^2(\nu(G)/ W_0, \mathbb{Z}) = \Sigma^2(\X(G)/ W, \mathbb{Z})$ whenever the invariants are defined.

If $G$ is finitely presented by Theorem \ref{homo123} $\Sigma^2(\nu(G)) = \Sigma^2(\nu(G)/ W_0)$.

If $G$ is $FP_2$ then by Lemma \ref{abel123} and Theorem \ref{Npres} $\Sigma^2(\nu(G), \mathbb{Z}) = \Sigma^2(\nu(G)/ W_0, \mathbb{Z})$. We observe that here we can apply Lemma \ref{abel123} since $W_0/[W_0, W_0]$ is finitely generated.
\end{proof}

If we want to calculate $\Sigma^1(\nu(G))$ only under the assumption that $G$ is finitely generated we cannot assume that $W_0$ is finitely generated since $G$ in general is not of type $FP_2$. 

But we can follow the ideas from the proofs from Section \ref{section-sigma1}. To do so we need to define two groups in $\nu(G)$ that would play the roles of $D$ and $L$ from $\X(G)$. Here we set
$$L_{\nu}(G) \hbox{ as the subgroup of } \nu(G) \hbox{ generated by } \{g \overline{g}^{ -1}  ~| ~g \in G \}$$
 and $$D_{\nu}(G) = [G, \overline{G}] \hbox{ in }\nu(G).$$ Note that both are normal subgroups in $\nu(G)$ and that the defining relations of $\nu(G)$ imply that
\begin{equation} \label{nu-commutator-relation} [D_{\nu}(G), L_{\nu}(G)] = 1.\end{equation}
 Note that $\nu(G)/ D_{\nu}(G) \simeq G \times \overline{G} \simeq \X(G)/ D$. 
 
 \begin{lemma} \label{abel0} Let $G$ be a finitely generated group. Then $L_{\nu}(G)$ is finitely generated.
 \end{lemma} 
 
 \begin{proof}
1)  We will prove first that the abelianization of $L_{\nu}(G)$ is finitely generated. By \cite[Thm.~2.1.1]{Said} the group $$\mathcal{E}(G) = \langle G, \overline{G} ~| ~ [\langle g \overline{g}^{ -1} \rangle, \langle g \overline{g}^{ -1} \rangle] = 1 \rangle$$ is isomorphic to $Aug(\mathbb{Z} G) \rtimes G$. Here $Aug(\mathbb{Z} G)$ is the augmentation ideal of $\mathbb{Z} G$. We write an element of $Aug(\mathbb{Z} G) \rtimes G$ as $(\lambda, g)$ where $\lambda \in Aug(\mathbb{Z} G)$ and $g \in G$. The product is $(\lambda_1, g_1) (\lambda_2, g_2) = (\lambda_1 + g_1 \lambda_2, g_1 g_2)$ and the isomorphism $$\theta: \mathcal{E}(G) \to Aug(\mathbb{Z} G) \rtimes G$$ sends $g \in G$ to $(0,g)$ and $\overline{g} \in \overline{G}$ to $(g-1,g)$. Note that $\theta$ induces an isomorphism
\begin{equation} \label{equ-12} \nu(G)/ L_{\nu}(G)' \simeq (Aug(\mathbb{Z} G) / I) \rtimes G,\end{equation}
where $I$ is a left ideal of $Aug(\mathbb{Z} G)$ and the above isomorphism restricted to the abelianization of $L_{\nu}(G)$ is
$$L_{\nu}(G)/ L_{\nu}(G)' \simeq Aug(\mathbb{Z} G) / I. $$

By the definition of $\theta$
 $$\theta([g_1, \overline{g}_2]) = \theta(g_1^{-1}) \theta(\overline{ g}_2^{-1}) \theta(g_1) \theta(\overline{g}_2) = (0, g_1^{ -1}) (g_2^{ -1} -1, g_2^{ -1} )(0, g_1)(g_2 -1, g_2) =$$$$ ( \alpha(g_1, g_2), [g_1, g_2]),$$
 where
 $$\alpha(g_1, g_2) = g_1^{ -1} (g_2^{ -1} -1) + g_1^{ -1} g_2^{ -1} g_1( g_2 - 1) = g_1^{ -1} g_2^{ -1} + [g_1, g_2] - g_1^{ -1}  - g_1^{ -1} g_2^{ -1} g_1.$$
 The relation
 $[g_1, \overline{g}_2]^{g_3} = [ g_1^{g_3}, \overline{g_2^{g_3}}]$ in $\nu(G)$, where $g_1, g_2 \in G, \overline{g}_2 \in \overline{G}$, implies 
 $$
  (g_3^{ -1} \alpha(g_1, g_2), [g_1^{ g_3} , g_2^{ g_3}] )= (0, g_3^{ -1} ) ( \alpha(g_1, g_2), [g_1, g_2]) ( 0, g_3) =$$ $$ \theta(g_3^{ -1}) \theta([g_1, \overline{g}_2]) \theta(g_3) = \theta([g_1, \overline{g}_2]^{g_3}) =$$ $$ \theta([ g_1^{g_3}, \overline{g_2^{g_3}}]) = ( \alpha(g_1^{g_3}, g_2^{g_3}), [g_1^{ g_3}, g_2^{ g_3}]) \hbox{ in } (Aug(\mathbb{Z} G)/ I) \rtimes G.$$
Then
 $$ g_3^{ -1} \alpha(g_1, g_2) = \alpha(g_1^{g_3}, g_2^{g_3}) = g_3^{ -1} \alpha(g_1, g_2) g_3 \hbox{ in } Aug(\mathbb{Z} G)/ I$$
 and so
 \begin{equation} \label{equ1} \alpha(g_1, g_2) ( g_3 -1) \in I.\end{equation}
 
Note that $$
\theta(g_3\overline{g}_3^{ -1}) =  (0, g_3) (g_3^{ -1} -1, g_3^{ -1}) = (1 - g_3, 1)$$
and
$$
\theta([g_1, \overline{g}_2]^{g_3 \overline{g}_3^{ -1}} ) = \theta(g_3 \overline{g}_3^{ -1})^{ -1} \theta([g_1, \overline{g}_2]) \theta({g_3 \overline{g}_3^{ -1}} )
=(g_3 -1, 1) (\alpha( g_1, g_2), [g_1, g_2]) ( 1 - g_3, 1) =$$ $$
(g_3 -1 + \alpha( g_1, g_2) + [g_1, g_2](1 - g_3), [g_1, g_2]).$$
The relation $[g_1, \overline{g}_2]^{g_3 \overline{g}_3^{ -1}} = [g_1, \overline{g}_2]$
  in $\nu(G)$ implies
 $$g_3 -1 + \alpha( g_1, g_2) + [g_1, g_2](1 - g_3) =  \alpha( g_1, g_2) \hbox{ in } Aug(\mathbb{Z} G)/I,$$
 hence
 \begin{equation} \label{equ2} ([g_1, g_2] - 1) ( g_3 -1) = 0 \hbox{ in } Aug(\mathbb{Z} G)/I.\end{equation}
 Since 
 $$
 \alpha(g_1, g_2) = g_1^{ -1} g_2^{ -1} + 1 - g_1^ { -1} - g_2^{ -1}  = (g_1^{ -1} -1) ( g_2^{ -1} - 1) \hbox{ in } Aug(\mathbb{Z} G)/Aug (\mathbb{Z} G ')
 $$
 by (\ref{equ1}), (\ref{equ2}) we get
 $$0 = \alpha(g_1, g_2) (g_3 -1) = (g_1^{ -1} -1) ( g_2^{ -1} - 1)(g_3 - 1) \hbox{ in } Aug(\mathbb{Z} G)/I.$$
 Thus for any $t_1, t_2, t_3 \in G$ we have
 \begin{equation} \label{equ3} (t_1 - 1) (t_2 - 1) ( t_3 - 1) \in I.\end{equation}
 Note that if $X$ is a finite generating set of $G$ then $Aug(\mathbb{Z} G)$ as a $\mathbb{Z}$-module is generated by $\cup_{k \geq 1} Y_k$, where $Y_k = \{ (x_{i_1}-1) \ldots (x_{i_k} -1) ~| x_{i_1}, \ldots, x_{i_k} \in X \cup X^{ -1}  \}$. Using (\ref{equ3}) temos que 
 $Aug(\mathbb{Z} G) / I$ as $\mathbb{Z}$-module is generated by 
the finite set  $Y_1 \cup Y_2$.
 
2) Finally we prove that $L_{\nu}(G)$ is finitely generated. Since we have a central extension $1 \to \Delta \to L_{\nu}(G) \to L_{\nu}(G)/ \Delta \to 1$ by  \cite[Lemma 2.2]{BK}  $L_{\nu}(G)$ is finitely generated if $L_{\nu}(G)/ \Delta$ is finitely generated and  the abelianization of $L_{\nu}(G)$ is finitely generated.  Now the isomorphism (\ref{equ0}) induces an isomorphism $L_{\nu}(G)/ \Delta \simeq L/ R$ and since $L$ is finitely generated \cite[Prop.~2.3]{BK}, $L_{\nu}(G)/ \Delta$ is finitely generated too. 
 \end{proof}

 \begin{prop} Let $G$ be a finitely generated group. We identify $S(\nu(G))$ with $S(\nu(G)/ W_0)$ via the canonical projection map $\nu(G) \to \nu(G)/ W_0$ and we identify $S(\nu(G)/ W_0)$ with $S(\X(G)/ W)$ via the isomorphism (\ref{iso321}).  Then $$\Sigma^1(\nu(G), \mathbb{Z}) = \Sigma^1(\X(G)/ W, \mathbb{Z}).$$
 
 {Remark} We recall that $\Sigma^1(\X(G)/ W, \mathbb{Z})$ was calculated in Section \ref{section-sigma1}.
 \end{prop}
 
 \begin{proof} I. By (\ref{nu-commutator-relation}) the same argument from Lemma \ref{lemma1.2} applies with $\X(G)$ substituted by $\nu(G)$ i.e. if  $\chi : \nu(G) \to \mathbb{R}$ is a character such that for $\chi_0 = (\chi_1, \chi_2) : G \times G \to \mathbb{R}$ defined by $\chi_1(g) = \chi(g), \chi_2(g) = \chi( \overline{g})$ we have that $\chi_1 \not= \chi_2$ and $[\chi_0] \in \Sigma^1(G \times G)$ then $[\chi] \in \Sigma^1(\nu(G))$. Recall that $[\chi_0] \in \Sigma^1(G \times G)$ is equivalent to one of the following  : 1) $\chi_1 = 0, [\chi_2] \in \Sigma^1(G)$; 2) $\chi_2 = 0, [\chi_1] \in \Sigma^1(G)$; 3) $\chi_1 \not= 0, \chi_2 \not= 0$.

II.  In Lemma \ref{abel0} we proved that $L_{\nu}(G)$ is a finitely generated subgroup of $\nu(G)$.
Then by Lemma \ref{lemma1.4} applied for $H = \nu(G)$ and $N = L_{\nu}(G)$  we deduce that if $\chi_1 = \chi_2 \not= 0$ and $[\chi_1] \in \Sigma^1(G)$ then $[\chi] \in \Sigma^1(\nu(G))$. 

I. and II. together with the description of $\Sigma^1(\X(G)/ W)$ in Section \ref{section-sigma1} imply that
$\Sigma^1(\X(G)/ W) \subseteq \Sigma^1(\nu(G))$.

On the other hand since $W_0 \subseteq \nu(G)'$ we can apply Lemma \ref{quotient-sigma1} and deduce
$$\Sigma^1(\nu(G) ) \subseteq \Sigma^1(\nu(G)/ W_0) = \Sigma^1(\X(G) / W),$$
where the last equality follows from the isomorphism  (\ref{iso321}).
\end{proof}

\end{document}